\documentclass[12pt]{article}
\usepackage{amsmath, amsthm, amsfonts, amssymb}
\usepackage{mathrsfs}
\usepackage{bbm}
\usepackage{amssymb}
\usepackage{txfonts}
\usepackage{graphicx}
\usepackage{wrapfig}

\newlength{\defbaselineskip}
\newcounter{marnote}

\newcommand{\setlinespacing}[1]%
           {\setlength{\baselineskip}{#1 \defbaselineskip}}

\theoremstyle{plain}
\newtheorem{theorem}{Theorem}[section]
\newtheorem{corollary}[theorem]{Corollary}
\newtheorem{lemma}[theorem]{Lemma}
\newtheorem{prop}[theorem]{Proposition}

\newtheorem{lemmaAlph}{Lemma}

\theoremstyle{definition}

\theoremstyle{remark}

\numberwithin{equation}{section}

\begin{document}


\title{Gradient estimates for solutions of the Lam\'{e} system with partially infinite coefficients}

\author{JiGuang Bao\footnote{School of Mathematical Sciences,
Beijing Normal University, Laboratory of Mathematics
and Complex Systems, Ministry of
   Education, Beijing 100875,
China. Email: jgbao@bnu.edu.cn.}\quad HaiGang
Li\footnote{
School of Mathematical Sciences,
Beijing Normal University, Laboratory of Mathematics
and Complex Systems, Ministry of
   Education, Beijing 100875,
China.
 Corresponding author. Email:
hgli@bnu.edu.cn.}\quad and\quad YanYan Li\footnote{Department of
Mathematics, Rutgers University, 110 Frelinghuysen Rd, Piscataway,
NJ 08854, USA. Email: yyli@math.rutgers.edu.} }

\date{}

\maketitle

\begin{abstract}
We establish upper bounds on the blow up rate
 of the gradients of solutions
of the Lam\'e system with partially infinite coefficients in
dimension two as the distance
between the surfaces of discontinuity of the coefficients
of the system tends to zero.
\end{abstract}

\setcounter {section} {0}

\section{Introduction}

We consider the Lam\'{e} system in linear elasticity.
 Let $\Omega\subset\mathbb{R}^{d}$, $d\geq2$,
 be a bounded open set with $C^{2}$ boundary,
and
$D_{1}$ and $D_{2}$ be
 two disjoint strictly
convex open sets in $\Omega$ with
$C^{2,\gamma}$ boundaries, $0<\gamma<1$, which are $\epsilon$-distance
 apart and far away
from $\partial{\Omega}$.  More precisely,
\begin{equation}\label{omega}
\begin{array}{l}
\displaystyle \overline{D}_{1},\overline{D}_{2}\subset\Omega,\quad\mbox{the principle curvatures of }\partial{D}_{1},\partial{D}_{2}\geq\kappa_{0}>0,\\
\displaystyle \epsilon:=\mathrm{dist}(D_{1},D_{2})>0,\quad \mathrm{dist}(D_{1}\cup{D}_{2},\partial{\Omega})>\kappa_{1}>0,
\end{array}
\end{equation}
where $\kappa_{0},\kappa_{1}$ are constants independent of $\epsilon$.

Denote
$$\widetilde{\Omega}:=\Omega\setminus\overline{D_{1}\cup{D}_{2}}.
$$
We assume that $\widetilde{\Omega}$ and $D_{1}\cup{D}_{2}$ are occupied
 by two different homogeneous  and isotropic  materials
 with different Lam\'{e} constants $(\lambda,\mu)$ and $(\lambda_{1},\mu_{1})$. Then the elasticity tensors for the inclusions and the background can be written, respectively, as  $\mathbb{C}^{1}$ and $\mathbb{C}^{0}$, with
$$C_{ij\,kl}^{1}=\lambda_{1}\delta_{ij}\delta_{kl}+\mu_{1}(\delta_{ik}\delta_{jl}+\delta_{il}\delta_{jk}),$$
and
$$C_{ij\,kl}^{0}=\lambda\delta_{ij}\delta_{kl}+\mu(\delta_{ik}\delta_{jl}+\delta_{il}\delta_{jk}),$$
where $i,j,k,l=1,2,\cdots,d$ and $\delta_{ij}$ is the Kronecker symbol: $\delta_{ij}=0$ for $i\neq{j}$, $\delta_{ij}=1$ for $i=j$.

Let $u=\left(u^1,u^2,\cdots,u^d\right)^{T}:~\Omega\rightarrow\mathbb{R}^{d}$ denote the displacement field.
For a given vector-valued function $\varphi$, we consider the following Dirichlet problem
\begin{equation}\label{system}
\begin{cases}
\nabla\cdot\bigg(\left(\chi_{\widetilde{\Omega}}\mathbb{C}^{0}+\chi_{D_{1}\cup{D}_{2}}\mathbb{C}^{1}\right)e(u)\bigg)=0,
&\mbox{in}~\Omega,\\
u=\mathbf{\varphi},&\mbox{on}~\partial{\Omega},
\end{cases}
\end{equation}
where $\chi_{D}$ is the characteristic function of $D$,
$$e(u):=\frac{1}{2}\left(\nabla{u}+(\nabla{u})^{T}\right)$$
is the strain tensor.

We assume that the standard ellipticity
 condition holds for \eqref{system}, that is,
$$
\mu>0,\quad\,d\lambda+2\mu>0;\quad\quad\mu_{1}>0,\quad\,d\lambda_{1}+2\mu_{1}>0.
$$

For $\varphi\in{H}^{1}(\Omega;\mathbb{R}^{d})$, it is well known that there exists a unique solution $u\in{H}^{1}(\Omega;\mathbb{R}^{d})$ of the Dirichlet problem \eqref{system}, which is also the minimizer of the energy functional
$$J[u]=\frac{1}{2}\int_{\Omega}
\bigg(\left(\chi_{\widetilde{\Omega}}\mathbb{C}^{0}+\chi_{D_{1}\cup{D}_{2}}\mathbb{C}^{1}\right)e(u),e(u)\bigg)dx$$
on
$${H}_{\varphi}^{1}(\Omega;\mathbb{R}^{d}):=\left\{~u\in{H}^{1}(\Omega;\mathbb{R}^{d})~\big|~u-\varphi\in{H}_{0}^{1}(\Omega;\mathbb{R}^{d})~\right\}.$$

Babu\v{s}ka, Andersson, Smith, and Levin \cite{ba} computationally analyzed the damage and fracture in fiber composite materials where the
Lam\'{e} system
is used. They observed numerically that
 the size of the strain tensor $e(u)$ remains
 bounded when the distance $\epsilon$ tends to zero.
Stimulated by this, there have
 been many works on the analogous question for the scalar equation
\begin{equation}\label{equk}
\begin{cases}
\nabla\cdot\Big(a_{k}(x)\nabla{u}_{k}\Big)=0&\mbox{in}~\Omega,\\
u_{k}=\varphi&\mbox{on}~\partial\Omega,
\end{cases}
\end{equation}
where $\varphi$ is given, and
$$a_{k}(x)=
\begin{cases}
k\in(0,\infty)&\mbox{in}~D_{1}\cup{D}_{2},\\
1&\mbox{in}~\widetilde{\Omega}.
\end{cases}
$$

For touching disks
 $D_{1}$ and $D_{2}$ in dimension $d=2$, Bonnetier and Vogelius \cite{bv} proved that $|\nabla{u}_{k}|$ remains bounded. The bound depends on the value of $k$. Li and Vogelius  \cite{lv} extended the result to general divergence form second order elliptic equations with piecewise smooth coefficients in all dimensions, and they proved that $|\nabla{u}|$ remains bounded as $\epsilon\rightarrow0$.
They also established  stronger,
$\epsilon$-independent,  $C^{1,\alpha}$ estimates
for solutions in the closure of each of the regions
$D_1$, $D_2$ and $\widetilde \Omega$.
 This extension covers domains $D_{1}$ and $D_{2}$ of
arbitrary smooth shapes.
 Li and Nirenberg extended in \cite{ln}
 the results in
\cite{lv} to general divergence form  second order elliptic systems including systems
of elasticity.  This in particular answered
in the affirmative the question naturally led to by
the above mentioned numerical indication
  in \cite{ba}
 for the boundedness of the strain tensor as
$\epsilon$ tends to $0$.
For higher derivative estimates, we draw attention of readers
to the open problem on page 894 of
\cite{ln}.

The estimates in \cite{ln} and \cite{lv} depend on
the ellipticity of the coefficients.
If ellipticity constants are allowed to deteriorate, the situation is
very different.  It  was  shown in various papers,
see for example Budiansky and Carrier
 \cite{bc} and  Markenscoff
\cite{m}, that when $k=
\infty$ in \eqref{equk} the $L^\infty$-norm of $|\nabla u_\infty|$
generally becomes unbounded
as $\epsilon$ tends to $0$.  The rate at which the
 $L^\infty$-norm  of the gradient of a special solution blows up
was shown in  \cite{bc}  to be $\epsilon^{-1/2}$ in dimension $d=2$.
Ammari, Kang and Lim \cite{akl}
  and
Ammari, Kang, Lee, Lee and Lim \cite{aklll}
proved that when $D_1$ and $D_2$ are disks in
$\mathbb{R}^{2}$, and when $k=\infty$ in  \eqref{equk},
the blow up rate of $|\nabla u_\infty|$ is
$\epsilon^{-1/2}$.  This result was extended by
Yun  \cite{y1,y2} and
Bao, Li and Yin \cite{bly1} to
strictly convex $D_1$ and $D_2$  in
$\mathbb{R}^{2}$.  In dimension $d=3$ and $d\ge 4$,
the  blow up rate of $|\nabla u_\infty|$ turns out to be
 $(\epsilon|\ln\epsilon|)^{-1}$ and $\epsilon^{-1}$
respectively; see \cite{bly1}.
The results were
 extended to multi-inclusions in
\cite{bly2}.
Further, more detailed,
 characterizations of the singular behavior of $\nabla{u}_{\infty}$
have been obtained by
Ammari, Ciraolo, Kang, Lee and Yun \cite{ackly},
Ammari, Kang, Lee, Lim and Zribi \cite{AKLLZ},
Bonnetier and Triki \cite{bt0, bt},
 Kang, Lim and
Yun \cite{kly, kly2}.
For related works, see
 \cite{adkl, agkl,  bt, byl,  FKNN, gn, keller1, keller2,
 Li-Li,
llby, ly,
ly2, MNT, x, BaoXiong} and the references therein.

In this paper we  obtain gradient estimates
for the Lam\'{e} system with infinity coefficients in
dimension $d=2$.
In  a subsequent paper we treat higher dimensional cases
$d\ge 3$.

The linear space of rigid displacements in $\mathbb{R}^{2}$ is
$$\Psi:=\bigg\{\psi\in{C}^{1}(\mathbb{R}^{2};\mathbb{R}^{2})~\big|~\nabla\psi+(\
\nabla\psi)^{T}=0~\bigg\},$$
or equivalently \cite{osy},
$$\Psi=\mathrm{span}
\left\{~\psi^{1}=
\begin{pmatrix}
 1 \\
 0 \\
\end{pmatrix},~\psi^{2}=\begin{pmatrix}
 0 \\
 1 \\
\end{pmatrix},~\psi^{3}=\begin{pmatrix}
x_{2} \\
-x_{1} \\
\end{pmatrix}~\right\}.
$$

If $\xi\in{H}^{1}(D; \mathbb{R}^{2})$, $e(\xi)=0$ in $D$,
and $D\subset \mathbb{R}^{2}$ is a
connected  open set,
 then $\xi$ is a linear combination of $\{\psi^{\alpha}\}$ in $D$. If an element $\xi$ in $\Psi$ vanishes at two distinct points of $\mathbb{R}^{2}$, then $\xi\equiv0$.

For fixed $\lambda$ and $\mu$ satisfying
$\mu>0$ and $\lambda+\mu>0$, denote $u_{\lambda_1, \mu_1}$ the solution
of (\ref{system}).
Then, as proved in the Appendix,
\begin{equation}
u_{\lambda_1, \mu_1}\to u \ \mbox{in}\
H^1(\Omega; \mathbb{R}^{2})\ \
\mbox{as}\ \min\{\mu_1, \lambda_1+\mu_1\}\to \infty.
\label{limitcase}
\end{equation}
where $u$ is a $H^1(\Omega;\mathbb{R}^{2})$ solution of
\begin{equation}
\label{mainequation}
\begin{cases}
\mathcal{L}_{\lambda,\mu}u:=\nabla\cdot\left(\mathbb{C}^{0}e(u)\right)=0,&\mbox{in}~\widetilde{\Omega},\\
u\big|_{+}=u\big|_{-},&\mbox{on}~\partial{D}_{1}\cup\partial{D}_{2},\\
e(u)=0,&\mbox{in}~D_{1}\cup{D}_{2},\\
\int_{\partial{D}_{i}}\frac{\partial{u}}{\partial\nu_{0}}\bigg|_{+}\cdot\psi^{\alpha}=0,&
\alpha=1,2,3, ~i=1,2,\\
u=\varphi,&\mbox{on}~\partial{\Omega},
\end{cases}
\end{equation}
where
$$\frac{\partial{u}}{\partial\nu_{0}}\bigg|_{+}:=
\left(\mathbb{C}^{0}e(u)\right)\vec{n}
=\lambda\left(\nabla\cdot{u}\right)\vec{n}
+\mu\left(\nabla{u}+(\nabla{u})^{T}\right)\vec{n},
$$
and $\vec{n}$ is the unit outer normal of $D_i$,
$i=1,2$.

Here and throughout this paper the subscript $\pm$ indicates the limit from outside and inside the domain, respectively.
The existence, uniqueness and     regularity
  of weak solutions
to \eqref{mainequation} are
proved  in the Appendix.
In particular,
 the  $H^{1}$ weak solution to \eqref{mainequation} is
in
 ${C}^{1}(\overline{\widetilde{\Omega}})
\cap{C}^{1}(\overline{D_{1}\cup{D}_{2}})$.

The convergence (\ref{limitcase}) in the case
$\mu_1\to \infty$ while $\lambda_1$ remains bounded was
established in \cite{akkl}.
Our proof of (\ref{limitcase}) in the Appendix is different and is an
extension to systems
 of that in  \cite{bly1}.

The solution of \eqref{mainequation} is also the unique
function which has the least energy in appropriate functional spaces,
characterized by
$$
I_{\infty}[u]=\min_{v\in\mathcal{A}}I_{\infty}[v],
$$
where
$$
I_{\infty}[v]:=
\frac{1}{2}\int_{\widetilde \Omega}\left(\mathbb{C}^{(0)}
e(v),e(v)\right)dx,
$$
and
$$
\mathcal{A}:=\left\{u\in{H}_{\varphi}^{1}(\Omega;\mathbb{R}^{2})
~\big|~e(u)=0~~\mbox{in}~D_1\cup D_2\right\}.
$$

 A calculation gives
\begin{align}\label{L_u}
\left(\mathcal{L}_{\lambda,\mu}u\right)^{i}=\mu\Delta{u}^{i}+
(\lambda+\mu)
\left[\partial_{x_ix_1}u^1+ \partial_{x_ix_2}u^2\right], \qquad
i=1,2.
\end{align}

Since $D_{1}$ and $D_{2}$ are two strictly convex subdomains of $\Omega$, there exist two points $P_{1}\in\partial{D}_{1}$ and $P_{2}\in\partial{D}_{2}$ such that
\begin{equation}\label{P1P2}
\mathrm{dist}(P_{1},P_{2})=\mathrm{dist}(\partial{D}_{1},\partial{D}_{2})=\epsilon.
\end{equation}
We use $\overline{P_{1}P_{2}}$ to denote the line segment connecting $P_{1}$ and $P_{2}$. For readers' convenience, we first assume that $\partial{D}_{1}$ near $P_{1}$ and $\partial{D}_{2}$ near $P_{2}$ are quadratic. For more general $D_{1}$ and $D_{2}$, we consider in Section 5.

 Assume that for some $\delta_{0}>0$,
\begin{equation}\label{coeff4_strongelyconvex}
\delta_{0}\leq\mu,\lambda+\mu\leq\frac{1}{\delta_{0}}.
\end{equation}
The main result in  this paper is as follows.
\begin{theorem}\label{mainthm1}
Assume that
 $\Omega$, $D_{1},D_{2}$, $\epsilon$
are defined  in \eqref{omega} with $d=2$,
$\lambda$ and $\mu$ satisfy
(\ref{coeff4_strongelyconvex}), and
 $\varphi\in{C}^{1,\gamma}(\partial\Omega;\mathbb{R}^{2})$
for some  $0<\gamma<1$. Let $u\in{H}^{1}(\Omega;\mathbb{R}^{2})\cap{C}^{1}(\overline{\widetilde{\Omega}};\mathbb{R}^{2})$ be
a solution to \eqref{mainequation}. Then for $0<\epsilon<1$,
 we have
\begin{equation}\label{mainestimates}
|\nabla{u}(x)|\leq
\left\{
\begin{array}{ll}
\displaystyle{
\frac{C}{\sqrt{\epsilon}+\mathrm{dist}(x,\overline{P_{1}P_{2}})}\|\varphi\|_{C^{1,\gamma}
(\partial{\Omega};\mathbb{R}^{2})},
}
&\quad\,x\in
\widetilde \Omega,\\
&\\
C\|\varphi\|_{C^{1,\gamma}
(\partial{\Omega};\mathbb{R}^{2})},&\quad x\in D_1\cup D_2.
\end{array}
\right.
\end{equation}
where $C$ is a universal constant.
In particular,
\begin{equation}\label{normbound}
\|\nabla u\|_{ L^\infty(\Omega) }\le C\epsilon^{ -1/2}
\|\varphi\|_{C^{1,\gamma}(\partial{\Omega};\mathbb{R}^{2})}.
\end{equation}
\end{theorem}
Note that throughout the paper,
unless otherwise stated, $C$ denotes some constant,
whose value may vary from line to line,
depending only on
 $\kappa_{0},\kappa_{1}$, $\gamma$, $\delta_{0}$,
 $\|\partial{D}_{1}\|_{C^{2,\gamma}}$,
 $\|\partial{D}_{2}\|_{C^{2,\gamma}}$,
  $\|\partial{\Omega}\|_{C^{2}}$ and the
Lebesgue measure of $\Omega$, and is in particular
 independent of $\epsilon$.  Also, we call a
constant having such dependence a universal constant.

Since the blow up rate of $|\nabla u_\infty|$ for solutions
of (\ref{equk}) when $k=\infty$ is known to reach the magnitude
 $\epsilon^{-1/2}$,
estimate (\ref{normbound}) is expected to be
optimal.  This is also supported by
 the numerical indication in \cite{kk}.

The paper is organized as follows.
In Section 2, we first introduce the setup of the
proof of Theorem \ref{mainthm1}. Then
we  state a proposition,
Proposition \ref{prop_gradient},
 containing  key estimates, and
deduce  Theorem \ref{mainthm1} from the proposition.
In Sections 3 and 4, we prove Proposition \ref{prop_gradient}.
In Section 5, we prove Theorem \ref{mainthm1*} which
extends Theorem \ref{mainthm1} in two aspects.
  One is that
the strict convexity assumption on $\partial D_1$ and $\partial D_2$
can be replaced by a weaker relative strict convexity assumption.
The other  is an upper bound of the gradient when
 the flatness order near the closest points
between  $\partial D_1$ and $\partial D_2$
is $m\ge 2$ instead of $m=2$ for the strictly convex
$\partial D_1$ and $\partial D_2$.
In the Appendix, we give a variational characterization
of solutions of the Lam\'e system with infinity coefficients
and prove the previously mentioned
 convergence result (\ref{limitcase}).

\section{Outline of the proof of Theorem \ref{mainthm1} and recall of Korn's inequalities }

The proof of Theorem \ref{mainthm1} makes use of the
following decomposition.
By the third line of \eqref{mainequation},
 $u$ is a linear combination of $\{\psi^{\alpha}\}$ in $D_{1}$ and $D_{2}$,
 respectively.
 Since
 $\mathcal{L}_{\lambda,\mu}\xi=0$ in $\widetilde{\Omega}$
 and $\xi=0$ on $\partial\widetilde{\Omega}$
 imply that $\xi=0$ in $\widetilde{\Omega}$, we decompose the solution of \eqref{mainequation}, in the spire of \cite{bly1},  as follows:
\begin{equation}\label{decom_u}
u=
\begin{cases}
\sum\limits_{\alpha=1}^{3}C_{1}^{\alpha}\psi^{\alpha},&\mbox{in}~\overline{D}_{1},\\
\sum\limits_{\alpha=1}^{3}C_{2}^{\alpha}\psi^{\alpha},&\mbox{in}~\overline{D}_{2},\\
\sum\limits_{\alpha=1}^{3}C_{1}^{\alpha}v_{1}^{\alpha}+\sum\limits_{\alpha=1}^{3}C_{2}^{\alpha}v_{2}^{\alpha}+v_{3},&\mbox{in}~\widetilde{\Omega},
\end{cases}
\end{equation}
where
$v_{i}^{\alpha}\in{C}^{1}(\overline{\widetilde{\Omega}};\mathbb{R}^{2})
\cap {C}^{2}(\widetilde{\Omega}; \mathbb{R}^{2})$, $\alpha=1,2,3$, $i=1,2$, satisfy
\begin{equation}\label{v1alpha}
\begin{cases}
\mathcal{L}_{\lambda,\mu}v_{i}^{\alpha}=0,&\mbox{in}~\widetilde{\Omega},\\
v_{i}^{\alpha}=\psi^{\alpha},&\mbox{on}~\partial{D}_{i},\\
v_{i}^{\alpha}=0,&\mbox{on}~\partial{D}_{j}\cup\partial{\Omega},~j\neq\,i;
\end{cases}
\end{equation}
 $v_{3}\in{C}^{1}(\overline{\widetilde{\Omega}};\mathbb{R}^{2})
\cap {C}^{2}(\widetilde{\Omega}; \mathbb{R}^{2})$ satisfies
\begin{equation}\label{v3}
\begin{cases}
\mathcal{L}_{\lambda,\mu}v_{3}=0,&\mbox{in}~\widetilde{\Omega},\\
v_{3}=0,&\mbox{on}~\partial{D}_{1}\cup\partial{D}_{2},\\
v_{3}=\varphi,&\mbox{on}~\partial{\Omega};
\end{cases}
\end{equation}
and the constants $\{C_{i}^{\alpha}\}$
are uniquely determined by $u$.

By the decomposition \eqref{decom_u}, we write
\begin{equation}\label{nablau_dec}
\nabla\,u=\sum_{\alpha=1}^{2}\left(C_{1}^{\alpha}-C_{2}^{\alpha}\right)\nabla{v}_{1}^{\alpha}
+\sum_{\alpha=1}^{2}C_{2}^{\alpha}(\nabla{v}_{1}^{\alpha}+\nabla{v}_{2}^{\alpha})
+\sum_{i=1}^{2}C_{i}^{3}\nabla{v}_{i}^{3}+\nabla{v}_{3},\quad\mbox{in}~\widetilde{\Omega}.
\end{equation}

Theorem \ref{mainthm1} can be deduced from the following proposition.

\begin{prop}\label{prop_gradient}
Under the hypotheses of Theorem \ref{mainthm1}
and a normalization
 $\|\varphi\|_{C^{1,\gamma}(\partial\Omega)}=1$,
we have, for $0<\epsilon<1$,
\begin{equation}\label{mainev3}
\|\nabla{v}_{3}\|_{L^{\infty}(\widetilde{\Omega})}\leq\,C;
\end{equation}
\begin{equation}\label{mainev1+23}
\|\nabla{v}_{1}^{\alpha}+\nabla{v}_{2}^{\alpha}\|_{L^{\infty}(\widetilde{\Omega})}\leq\,C,\quad\alpha=1,2, 3;
\end{equation}
\begin{equation}\label{mainev1}
|\nabla{v}_{i}^{\alpha}(x)|\leq\frac{C}{\epsilon+\mathrm{dist}^{2}(x,\overline{P_{1}P_{2}})},\quad\,i,\alpha=1,2,\quad\,x\in\widetilde{\Omega};
\end{equation}
\begin{equation}\label{mainevi3}
|\nabla{v}_{i}^{3}(x)|\leq\,C\frac{\epsilon+\mathrm{dist}(x,\overline{P_{1}P_{2}})}{\epsilon+\mathrm{dist}^{2}(x,\overline{P_{1}P_{2}})},\quad\,i=1,2,
\quad\,x\in\widetilde{\Omega};
\end{equation}
and
\begin{equation}\label{maineC}
|C_{i}^{\alpha}|\leq\,C,\quad
i=1,2,~\alpha=1,2,3;
\end{equation}
\begin{equation}\label{maineC1-C2}
\quad|C_{1}^{\alpha}-C_{2}^{\alpha}|\leq\,C\sqrt{\epsilon},\quad
\alpha=1,2.
\end{equation}
\end{prop}

\begin{proof} [Proof of
 Theorem \ref{mainthm1}
by using Proposition
\ref{prop_gradient}]

Clearly, we only
need to prove the theorem
under the normalization
 $\|\varphi\|_{C^{1,\gamma}(\partial\Omega)}=1$.

Since
$$\nabla{u}=C_{i}^{3}
\begin{pmatrix}
0&1\\
-1&0
\end{pmatrix}\quad\mbox{in}~D_{i},\quad\,i=1,2,$$
the second estimate in
(\ref{mainestimates}) follows
easily from (\ref{maineC}).

By \eqref{nablau_dec} and Proposition \ref{prop_gradient}, we have,
for $x$ in $\widetilde \Omega$,
$$
\left|\nabla\,u(x)\right|\leq\sum_{\alpha=1}^{2}\left|C_{1}^{\alpha}-C_{2}^{\alpha}\right|\left|\nabla{v}_{1}^{\alpha}(x)\right|
+C\sum_{i=1}^{2}\left|\nabla{v}_{i}^{3}(x)\right|+C
\leq\frac{C}{\sqrt{\epsilon}+\mathrm{dist}(x,\overline{P_{1}P_{2}})}.
$$
Theorem \ref{mainthm1}
follows.
\end{proof}

\bigskip

To complete this section, we recall some properties of the tensor $\mathbb{C}$.
For the isotropic elastic material, let
$$\mathbb{C}:=(C_{ij\,kl})=\left(\lambda\delta_{ij}\delta_{kl}+\mu\left(\delta_{ik}\delta_{jl}+\delta_{il}\delta_{jk}\right)\right),
\quad\,\mu>0,\quad\,d\lambda+2\mu>0.$$
The  components $C_{ij\,kl}$ satisfy the following symmetric condition:
\begin{equation}\label{coeff1}
C_{ij\,kl}=C_{kl\,ij}=C_{klj\,i},\quad\,i,j,k,l=1,2,\cdots,d.
\end{equation}
We will use the following notations:
$$(\mathbb{C}A)_{ij}=\sum_{k,l=1}^{d}C_{ij\,kl}A_{kl},\quad\mbox{and}\quad
(A,B)\equiv A:B =\sum_{i,j=1}^{d}A_{ij}B_{ij},$$
for every pair of $d\times{d}$ matrices $A=(A_{ij}),B=(B_{ij})$.
Clearly
$$
(\mathbb{C}A,B)=(A,\mathbb{C}B).
$$
If $A$ is  symmetric, then,
by  the symmetry condition \eqref{coeff1},
we have that
$$
\left(\mathbb{C} A, A\right)
=
C_{ij\,kl}\,A_{kl}A_{ij}=\lambda\,A_{ii}A_{kk}+2\mu\,A_{kj}A_{kj}.$$
Thus
 $\mathbb{C}$ satisfies the following ellipticity condition:
For every $d\times{d}$ real symmetric matrix $A=(A_{ij})$,
\begin{equation}\label{coeff2}
\min\{2\mu,d\lambda+2\mu\}|A|^{2}\leq\,
\left(\mathbb{C} A, A\right)
\leq\,
\max\{2\mu,d\lambda+2\mu\}|A|^{2},
\end{equation}
where $|A|^{2}=\sum\limits_{i,j}A_{ij}^{2}$.

For readers' convenience, we recall some inequalities of Korn's type, see, e.g. theorem 2.1, theorem 2.5, theorem 2.10 and theorem 2.14 in  \cite{osy}.
\begin{lemmaAlph}\label{lemmaA}(First Korn inequality)
Let $\Omega$ be a bounded open set of $\mathbb{R}^{d}$, $d\geq2$. Then every  $u\in{H}_{0}^{1}(\Omega,\mathbb{R}^{d})$ satisfies the inequality
$$\|\nabla{u}\|_{L^{2}(\Omega)}^{2}\leq2\|e(u)\|_{L^{2}(\Omega)}^{2}.$$
\end{lemmaAlph}
Next, a few versions of the Second Korn inequality
\begin{lemmaAlph}\label{lemmaB}
Suppose that $\Omega$ is a bounded open set of $\mathbb{R}^{d}$, $d\geq2$, of diameter $R$, and it is star-shaped with respect to the ball $B_{R_{1}}=\{x:~|x|<R_{1}\}$. Then for any $u\in{H}^{1}(\Omega,\mathbb{R}^{d})$ we have the inequality
$$\|\nabla{u}\|_{L^{2}(\Omega)}^{2}\leq\,C_{1}\left(\frac{R}{R_{1}}\right)^{d+1}\|e(u)\|_{L^{2}(\Omega)}^{2}
+C_{2}\left(\frac{R}{R_{1}}\right)^{d}\|\nabla{u}\|_{L^{2}(B_{R_{1}})}^{2},$$
where $C_{1},C_{2}$ are constants depending only on $d$.
\end{lemmaAlph}
We remark that the above inequality holds for a Lipschitz domain $\Omega$, with $C_{1}$ and $C_{2}$ depending on $\Omega$, since such a domain is a union of a finite number of star-shaped domains. The following lemma is an easy consequence of Lemma \ref{lemmaA} and Lemma \ref{lemmaB}.
\begin{lemmaAlph}\label{lemmaC}
Suppose that $\Omega$ satisfies the condition of Lemma \ref{lemmaB} and $u\in{H}^{1}(\Omega,\mathbb{R}^{d})$. Then
$$\|\nabla{u}\|_{L^{2}(\Omega)}^{2}\leq\,C_{1}\left(\frac{R}{R_{1}}\right)^{d+1}\|e(u)\|_{L^{2}(\Omega)}^{2}
+C_{2}\left(\frac{R}{R_{1}}\right)^{d}\gamma^{-2}\|u\|_{L^{2}(\Omega)}^{2},$$
where $\gamma$ is the distance of $B_{R_{1}}$ from $\partial\Omega$, and $C_{1},C_{2}$ depend only on $d$.
\end{lemmaAlph}
In applications it is often important to have the following version of the Second Korn inequality. We still use $\Psi$ to denote the linear space of rigid displacements in $\mathbb{R}^{d}$. Then
\begin{lemmaAlph}\label{lemmaD}
Let $\Omega$ be a bounded Lipschitz open set of $\mathbb{R}^{d}$, $d\geq2$, and let $V$ be a closed subspace of $H^{1}(\Omega,\mathbb{R}^{d})$, such that $V\cap\Psi=\{0\}$. Then every $v\in{V}$ satisfies
$$\|v\|_{H^{1}(\Omega)}\leq\,C\|e(v)\|_{L^{2}(\Omega)},$$
where $C$ depends only on $\Omega$ and $V$.
\end{lemmaAlph}

\section{Estimates of
 $\nabla{v}_{1}^{\alpha}$,
 $\nabla{v}_{2}^{\alpha}$ and $\nabla{v}_{3}$}\label{sec_gradient}

Before proceeding to prove Proposition \ref{prop_gradient}, we first fix notations. By a translation and rotation of the coordinates if necessary, we may assume without loss of generality that the points $P_{1}$ and $P_{2}$ in \eqref{P1P2} satisfy
$$P_{1}=\left(0,\frac{\epsilon}{2}\right)\in\partial{D}_{1},\quad\mbox{and}\quad\,P_{2}=\left(0,-\frac{\epsilon}{2}\right)\in\partial{D}_{2}.$$
 Fix
 a small universal constant $R$, such that the portions of
 $\partial{D}_{i}$  near $P_{i}$  can be represented
respectively by
$$x_{2}=\frac{\epsilon}{2}+h_{1}(x_{1}),\quad\mbox{and}\quad\,x_{2}=-\frac{\epsilon}{2}+h_{2}(x_{1}), \quad\mbox{for}~~|x_{1}|<
2R.$$
Moreover, by
the assumptions on $\partial{D}_{i}$,
  $h_{i}$  satisfies
$$\frac{\epsilon}{2}+h_{1}(x_{1})>-\frac{\epsilon}{2}+h_{2}(x_{1}),\quad\mbox{for}~~|x_{1}|<2R,$$
\begin{equation}\label{h1h20}
h_{1}(0)=h_{2}(0)=h'_{1}(0)=h'_{2}(0)=0,
\end{equation}
\begin{equation}\label{h1h22}
h_{1}^{''}(0)\geq\kappa_{0}>0,\quad\,h_{2}^{''}(0)\leq-\kappa_{0}<0,
\end{equation}
and
\begin{equation}\label{h1h2}
\|h_{1}\|_{C^{2,\gamma}([-2R,2R])}+\|h_{2}\|_{C^{2,\gamma}([-2R,2R])}\leq{C}.
\end{equation}
For $0<r\leq\,2R$, denote
\begin{equation*}
 \Omega_r:=\left\{x\in \mathbb{R}^{2}~\big|~-\frac{\epsilon}{2}+h_{2}(x_{1})<x_{2}<\frac{\epsilon}{2}+h_{1}(x_{1}),~|x_{1}|<r\right\}.
\end{equation*}
The top
 and bottom  boundaries of $\Omega_{r}$ are
$$
\Gamma_{r}^{+}=
\left\{x\in \mathbb{R}^{2}~\big|~x_{2}
=\frac{\epsilon}{2}+h_{1}(x_{1}),~|x_{1}|<r\right\},
$$
and
$$ \Gamma_{r}^{-}=
\left\{x\in\mathbb{R}^{2}~\big|~x_{2}=-\frac{\epsilon}{2}+h_{2}(x_{1}),~|x_{1}|<r\right\}.
$$
Here $x=(x_1, x_2)$.

\subsection{Estimates of $v_{3}$ and
$v_{1}^{\alpha}+v_{2}^{\alpha}$, $\alpha=1,2,3$}

\begin{lemma}\label{lem_v3_v1+v2}
$$
\|v_{3}\|_{L^{\infty}(\widetilde{\Omega})}+\|\nabla{v}_{3}\|_{L^{\infty}(\widetilde{\Omega})}\leq\,C.
$$
$$
\|v_{1}^{\alpha}+v_{2}^{\alpha}\|_{L^{\infty}(\widetilde{\Omega})}
+\|\nabla{v}_{1}^{\alpha}+\nabla{v}_{2}^{\alpha}\|_{L^{\infty}(\widetilde{\Omega})}\leq\,C,\quad\alpha=1,2,3.
$$
\end{lemma}

\begin{proof}
As mentioned before, we may assume without loss of generality that $\|\varphi\|_{C^{1,\gamma}(\partial\Omega)}=1$. Extending $\varphi$ to
$\Phi\in{C}^{1,\gamma}(\overline{\Omega})$
satisfying $\Phi(x)=0$ for all
$dist(x, \partial \Omega)>\kappa_1/2$.  In particular,
 $\Phi=0$
near $\overline{{D}_{1}\cup {D}_{2}}$, and
$$\int_{\widetilde{\Omega}}|\nabla\Phi|^{2}dx\leq\,C\|\varphi\|_{C^{1,\gamma}(\partial\Omega)}=C.$$
Then, in view of \eqref{v3},
$$I_{\infty}[v_{3}]:=\frac{1}{2}\int_{\widetilde{\Omega}}\left(\mathbb{C}^{0}e(v_{3}),e(v_{3})\right)dx\leq\,I_{\infty}[\Phi]\leq\,C.$$
By the First Korn inequality (Lemma \ref{lemmaA}) and \eqref{coeff2},
\begin{align*}
\left\|\nabla(v_{3}-\Phi)\right\|_{L^{2}(\widetilde{\Omega})}^{2}&\leq2\|e(v_{3}-\Phi)\|_{L^{2}(\widetilde{\Omega})}^{2}\\
&\leq\,C\left(\|e(v_{3})\|_{L^{2}(\widetilde{\Omega})}^{2}+\|e(\Phi)\|_{L^{2}(\widetilde{\Omega})}^{2}\right)\\
&\leq\,C\left(I_{\infty}[v_{3}]+I_{\infty}[\Phi]\right)\\
&\leq\,C.
\end{align*}
It follows that
$$\|\nabla{v}_{3}\|_{L^{2}(\widetilde{\Omega})}\leq\,C.$$
Consequently,
$$\|v_{3}\|_{L^{2}(\widetilde{\Omega})}\leq\,C\|\nabla{v}_{3}\|_{L^{2}(\widetilde{\Omega})}\leq\,C.$$
Note that the constant $C$ above is independent of $\epsilon$.
By the interior estimates and the boundary estimates for elliptic
systems (see Agmon, Douglis and Nirenberg
\cite{adn1} and \cite{adn}), we have
$$
\|\nabla{v}_{3}\|_{L^{\infty}(\widetilde{\Omega}\setminus\Omega_{R/2})}\leq\,C.
$$
We apply theorem 1.1 in \cite{llby} to $v_{3}$ and obtain
$$
\|\nabla{v}_{3}\|_{L^{\infty}(\Omega_{R/2})}\leq\,C.
$$

Since
\begin{equation*}\label{v1v2alpha}
\begin{cases}
\mathcal{L}_{\lambda,\mu}(v_{1}^{\alpha}+v_{2}^{\alpha}-\psi^{\alpha})=0,&\mbox{in}~\widetilde{\Omega},\\
v_{1}^{\alpha}+v_{2}^{\alpha}-\psi^{\alpha}=0,&\mbox{on}~\partial{D}_{1}\cup\partial{D}_{2},\\
v_{1}^{\alpha}+v_{2}^{\alpha}-\psi^{\alpha}=-\psi^{\alpha},&\mbox{on}~\partial{\Omega},
\end{cases}
\end{equation*}
the above arguments
yield, with $\varphi=-\psi^\alpha$,
\begin{equation}\label{nablav1+v2}
\left\|\nabla{v}_{1}^{\alpha}+\nabla{v}_{2}^{\alpha}\right\|_{L^{\infty}(\widetilde{\Omega})}\leq\,C,\quad\alpha=1,2,3.
\end{equation}
Lemma \ref{lem_v3_v1+v2} follows from the above.
\end{proof}

\subsection{Estimates of $v_{i}^{\alpha}$, $i,\alpha=1,2$}

To estimate $v_{i}^{\alpha}$, $i,\alpha=1,2$,
we introduce a scalar function $\bar{u}\in{C}^{2}(\mathbb{R}^{2})$, such that $\bar{u}=1$ on
$\partial{D}_{1}$, $\bar{u}=0$ on
$\partial{D}_{2}\cup\partial\Omega$,
\begin{align}\label{ubar}
\bar{u}(x)
=\frac{x_{2}-h_{2}(x_{1})+\frac{\epsilon}{2}}{\epsilon+h_{1}(x_{1})-h_{2}(x_{1})},\quad\mbox{in}~~\Omega_{R},
\end{align}
and
\begin{equation}\label{nablau_bar_outside}
\|\bar{u}\|_{C^{2}(\mathbb{R}^{2}\setminus \Omega_R)}\leq\,C.
\end{equation}
 A calculation gives
\begin{equation}\label{nablau_bar}
|\partial_{x_{1}}\bar{u}(x)|\leq\frac{C|x_{1}|}{\epsilon+|x_{1}|^{2}},\quad
|\partial_{x_{2}}\bar{u}(x)|\leq\frac{C}{\epsilon+|x_{1}|^{2}},\quad~x\in\Omega_{R},
\end{equation}
\begin{equation}\label{nabla2u_bar}
|\partial_{x_{1}x_{1}}\bar{u}(x)|\leq\frac{C}{\epsilon+|x_{1}|^{2}},\quad
|\partial_{x_{1}x_{2}}\bar{u}(x)|\leq\frac{C|x_{1}|}{(\epsilon+|x_{1}|^{2})^{2}},
\quad\partial_{x_{2}x_{2}}\bar{u}(x)=0,\quad~x\in\Omega_{R}.
\end{equation}

Define
\begin{equation}\label{def:ubar1112}
\bar{u}_{1}^{1}=(\bar{u},0)^{T}, \quad\bar{u}_{1}^{2}=(0,\bar{u})^{T},
\quad\,\mbox{in}~~\widetilde{\Omega},
\end{equation}
then $
v_{1}^{\alpha}=
\bar{u}_{1}^{\alpha}$ on $\partial\widetilde{\Omega}$.
Similarly, we can define
\begin{equation}\label{def:ubar2122}
\bar{u}_{2}^{1}=(\underline{u},0)^{T},\quad\,\bar{u}_{2}^{2}=(0,\underline{u})^{T},
\quad\,\mbox{in}~~\widetilde{\Omega},
\end{equation}
where $\underline{u}$ is a scalar function in ${C}^{2}(\mathbb{R}^{2})$ satisfying $\underline{u}=1$ on
$\partial{D}_{2}$, $\underline{u}=0$ on
$\partial{D}_{1}\cup\partial\Omega$,
\begin{align}\label{u_underline}
\underline{u}(x)
=\frac{-x_{2}+h_{1}(x_{1})+\frac{\epsilon}{2}}{\epsilon+h_{1}(x_{1})-h_{2}(x_{1})},\quad\,x\in\Omega_{R},
\end{align}
and
\begin{equation}\label{nablau_underline_outside}
\|\underline{u}\|_{C^{2}(\mathbb{R}^{2}\setminus \Omega_R)}\leq\,C.
\end{equation}
By \eqref{L_u}, \eqref{nablau_bar} and \eqref{nabla2u_bar},
\begin{equation}\label{L_ubar_ialpha}
\left|\mathcal{L}_{\lambda,\mu}\bar{u}_{i}^{\alpha}(x)\right|\leq\frac{C}{\epsilon+|x_{1}|^{2}}
+\frac{C|x_{1}|}{(\epsilon+|x_{1}|^{2})^{2}},\quad\,i,\alpha=1,2,\quad~x\in\Omega_{R}.
\end{equation}

For $|z_{1}|\leq\,R$, we always use $\delta$ to denote
\begin{equation}\label{delta}
\delta:=\delta(z_{1})=\frac{\epsilon+h_{1}(z_{1})-h_{2}(z_{1})}{2}.
\end{equation}
Clearly,
\begin{equation}
\frac{1}{C}(\epsilon+|z_{1}|^{2})\leq\delta(z_{1})\leq\,C(\epsilon+|z_{1}|^{2}).
\end{equation}
For $|z_{1}|\leq\,R/2$, $s<R/2$, let
\begin{equation}\label{hatomega}
\widehat{\Omega}_{s}(z_{1}):=\left\{(x_{1},x_{2})~\big|~-\frac{\epsilon}{2}+h_{2}(x_{1})<x_{2}<\frac{\epsilon}{2}+h_{1}(x_{1}),
~|x_{1}-z_{1}|<s\right\}.
\end{equation}
We denote
\begin{equation}\label{def_w}
w_{i}^{\alpha}:=v_{i}^{\alpha}-\bar{u}_{i}^{\alpha},\quad\,i,\alpha=1,2.
\end{equation}

In order to prove \eqref{mainev1}, it suffices to prove the following proposition.

\begin{prop}\label{prop1}
Assume the above, let $v_{i}^{\alpha}\in{C}^2(\widetilde{\Omega};\mathbb{R}^{2})\cap{C}^1(\overline{\widetilde{\Omega}};\mathbb{R}^{2})$ be the
weak solution of \eqref{v1alpha}. Then, for $i,\alpha=1,2$,
\begin{equation}\label{energy_w}
\int_{\widetilde{\Omega}}\left|\nabla{w}_{i}^{\alpha}\right|^{2}dx\leq\,C,
\end{equation}
\begin{equation}\label{energy_w_inomega_z1}
\int_{\widehat{\Omega}_{\delta}(z_{1})}\left|\nabla{w}_{i}^{\alpha}\right|^{2}dx\leq
\begin{cases}C(\epsilon^{2}+|z_{1}|^{2}),&|z_{1}|\leq\sqrt{\epsilon},\\
C|z_{1}|^{2},&\sqrt{\epsilon}<|z_{1}|\leq\,R,
\end{cases}
\end{equation}
and
\begin{equation}\label{nabla_w_ialpha}
|\nabla{w}_{i}^{\alpha}(x)|\leq
\begin{cases}C
\frac{\epsilon+|x_1|}{\epsilon},&|x_{1}|\leq\sqrt{\epsilon},\\
\frac{C}{|x_{1}|},&\sqrt{\epsilon}<|x_{1}|\leq\,R.
\end{cases}
\end{equation}
\end{prop}

\begin{corollary}\label{cor1}
For $i,\alpha=1,2$,
\begin{equation}\label{nablav1}
\left|\nabla{v}_{i}^{\alpha}(x)\right|\leq\frac{C}{\epsilon+\mathrm{dist}^{2}(x,\overline{P_{1}P_{2}})},
\quad~~x\in\widetilde{\Omega}.
\end{equation}
\end{corollary}

\begin{proof}[Proof of Corollary \ref{cor1}]
A consequence of \eqref{energy_w} is
$$\int_{\widetilde{\Omega}\setminus\Omega_{R/2}}\left|\nabla{v}_{i}^{\alpha}\right|^{2}dx
\leq2\int_{\widetilde{\Omega}\setminus\Omega_{R/2}}\left(\left|\nabla\bar{u}_{i}^{\alpha}\right|^{2}+\left|\nabla{w}_{i}^{\alpha}\right|^{2}\right)dx\leq\,C,$$
With this we can apply classical elliptic estimates to obtain
\begin{equation}\label{nablavialpha_outomega1/2}
\left\|\nabla{v}_{i}^{\alpha}\right\|_{L^{\infty}(\widetilde{\Omega}\setminus\Omega_{R})}\leq\,C,\quad\,i,\alpha=1,2.
\end{equation}
Under  assumption (\ref{omega}),
$$
\frac 1C(\epsilon+|x_1|^2)\le
\mathrm{dist}(x,\overline{P_{1}P_{2}})\le
C (\epsilon+|x_1|^2).
$$
Estimate \eqref{nablav1} in $\Omega_{R}$ follows from \eqref{nabla_w_ialpha} and the fact that
$$\left|\nabla\bar{u}_{i}^{\alpha}(x)\right|\leq\frac{C}{\epsilon+|x_{1}|^{2}},\quad\mbox{in}~~\Omega_{R}.$$
\end{proof}

\begin{proof}[Proof of Proposition \ref{prop1}] The iteration scheme we use in the proof is similar in spirit to that used in \cite{llby}. We only prove it for $i=\alpha=1$, since the same proof applies to the other cases. For simplicity, denote $w:=w_{1}^{1}$. We divide into three steps.

\noindent{\bf STEP 1.} Proof of \eqref{energy_w}.

By \eqref{def_w},
\begin{equation}\label{w20}
\begin{cases}
\mathcal{L}_{\lambda,\mu}w=-\mathcal{L}_{\lambda,\mu}\bar{u}_{1}^{1},&\mbox{in}~\widetilde{\Omega},\\
w=0,&\mbox{on}~\partial\widetilde{\Omega}.
\end{cases}
\end{equation}
Multiplying the equation in \eqref{w20} by $w$ and integrating by parts, we have
\begin{align}\label{integratingbyparts}
\int_{\widetilde{\Omega}}\left(\mathbb{C}^{0}e(w),e(w)\right)dx
=\int_{\widetilde{\Omega}}w\left(\mathcal{L}_{\lambda,\mu}\bar{u}_{1}^{1}\right)dx.
\end{align}
By the mean value theorem,
there exists $r_{0}\in(R/2,2R/3)$ such that
\begin{align}\label{fubini}
\int\limits_{\scriptstyle |x_{1}|=r_{0},\atop\scriptstyle
-\epsilon/2+h_{2}(x_{1})<x_{2}<\epsilon/2+h_{1}(x_{1})\hfill}|w|dx_{2}&
=\frac{6}{R}
\int\limits_{\scriptstyle R/2<|x_{1}|<2R/3,\atop\scriptstyle
-\epsilon/2+h_{2}(x_{1})<x_{2}<\epsilon/2+h_{1}(x_{1})\hfill}|w|dx\nonumber\\
&\leq\,C\int_{\Omega_{2R/3}\setminus\Omega_{R/2}}|\nabla
w|dx\nonumber\\
&\leq\,C\left(\int_{\widetilde{\Omega}}\left|\nabla{w}\right|^{2}dx\right)^{1/2}.
\end{align}
It follows from
\eqref{coeff2},  \eqref{integratingbyparts}  and
the First Korn inequality that
\begin{align}\label{energy_w_1}
&\int_{\widetilde{\Omega}}\left|\nabla{w}\right|^{2}dx\nonumber\\
&\leq\,2\int_{\widetilde{\Omega}}|e(w)|^{2}dx\nonumber\\
&\leq\,C\bigg|\int_{\Omega_{r_{0}}}w\left(\mathcal{L}_{\lambda,\mu}\bar{u}_{1}^{1}\right)dx\bigg|
+C\bigg|\int_{\widetilde{\Omega}\setminus\Omega_{r_{0}}}w\left(\mathcal{L}_{\lambda,\mu}\bar{u}_{1}^{1}\right)dx\bigg|\nonumber\\
&\leq\,C\bigg|\int_{\Omega_{r_{0}}}w\left(\mathcal{L}_{\lambda,\mu}\bar{u}_{1}^{1}\right)dx\bigg|
+C\int_{\widetilde{\Omega}\setminus\Omega_{r_{0}}}|w|dx\nonumber\\
&\leq\,C\left(\bigg|\int_{\Omega_{r_{0}}}w^{(1)}\partial_{x_{1}x_{1}}\bar{u}\,dx\bigg|
+\bigg|\int_{\Omega_{r_{0}}}w^{(2)}\partial_{x_{1}x_{2}}\bar{u}\,dx\bigg|\right)
+C\left(\int_{\widetilde{\Omega}\setminus\Omega_{r_{0}}}|\nabla{w}|^{2}dx\right)^{1/2}.
\end{align}

First,
\begin{align*}
\int_{\Omega_{r_{0}}}w^{(1)}\partial_{x_{1}x_{1}}\bar{u}\,dx&=
-\int_{\Omega_{r_{0}}}\partial_{x_{1}}w^{(1)}\partial_{x_{1}}\bar{u}\,dx+\int\limits_{\scriptstyle |x_{1}|=r_{0},\atop\scriptstyle
-\epsilon/2+h_{2}(x_{1})<x_{2}<\epsilon/2+h_{1}(x_{1})\hfill}\left(\partial_{x_{1}}\bar{u}\right)w^{(1)}\,dx_{2}\\
&:=I+II.
\end{align*}
Then, by \eqref{nablau_bar},
\begin{align*}
|I|&\leq\,C\left(\int_{\Omega_{r_{0}}}|\partial_{x_{1}}\bar{u}|^{2}dx\right)^{1/2}
\left(\int_{\widetilde{\Omega}}|\nabla{w}|^{2}dx\right)^{1/2}\leq\,C\left(\int_{\widetilde{\Omega}}|\nabla{w}|^{2}dx\right)^{1/2}.
\end{align*}
By \eqref{fubini}, we have
$$|II|\leq\,C\int\limits_{\scriptstyle |x_{1}|=r_{0},\atop\scriptstyle
-\epsilon/2+h_{2}(x_{1})<x_{2}<\epsilon/2+h_{1}(x_{1})\hfill}|w|\,dx_{2}\leq\,C\left(\int_{\widetilde{\Omega}}|\nabla{w}|^{2}dx\right)^{1/2}.$$
Hence
\begin{equation}\label{lem:energy_w1}
\bigg|\int_{\Omega_{r_{0}}}w^{(1)}\partial_{x_{1}x_{1}}\bar{u}\,dx\bigg|\leq\,C\left(\int_{\widetilde{\Omega}}|\nabla{w}|^{2}dx\right)^{1/2}.
\end{equation}
Similarly, using $w=0$ on $\partial{D}_{1}\cup\partial{D}_{2}$,
\begin{align*}
\bigg|\int_{\Omega_{r_{0}}}w^{(2)}\partial_{x_{1}x_{2}}\bar{u}\,dx\bigg|&=
\bigg|\int_{\Omega_{r_{0}}}\partial_{x_{2}}w^{(2)}\partial_{x_{1}}\bar{u}\,dx\bigg|\\
&\leq\,C\left(\int_{\Omega_{r_{0}}}|\partial_{x_{1}}\bar{u}|^{2}\,dx\right)^{1/2}
\left(\int_{\widetilde{\Omega}}|\nabla{w}|^{2}\,dx\right)^{1/2}\\
&\leq\,C\left(\int_{\widetilde{\Omega}}|\nabla{w}|^{2}\,dx\right)^{1/2}.
\end{align*}
Therefore, combining this estimate with \eqref{lem:energy_w1} and \eqref{energy_w_1},
$$\int_{\widetilde{\Omega}}|\nabla{w}|^{2}\,dx\leq\,C\left(\int_{\widetilde{\Omega}}|\nabla{w}|^{2}\,dx\right)^{1/2},$$
which implies \eqref{energy_w}.

\noindent{\bf STEP 2.} Proof of \eqref{energy_w_inomega_z1}.

For $0<t<s<R$, let $\eta$ be a smooth function satisfying $\eta(x_{1})=1$ if $|x_{1}-z_{1}|<t$, $\eta(x_{1})=0$ if $|x_{1}-z_{1}|>s$, $0\leq\eta(x_{1})\leq1$ if $t\leq|x_{1}-z_{1}|\leq\,s$, and $|\eta'(x_{1})|\leq\frac{2}{s-t}$. Multiplying the equation in \eqref{w20} by $w\eta^{2}$ and integrating by parts
lead  to
\begin{equation}\label{I3-1}
\int_{\widehat{\Omega}_{s}(z_{1})}(\mathbb{C}^{0}e(w),e(w\eta^{2}))dx
=-\int_{\widehat{\Omega}_{s}(z_{1})}(w\eta^{2})\mathcal{L}_{\lambda,\mu}\bar{u}_{1}^{1}\,dx.
\end{equation}
Using the First Korn inequality and some standard arguments, we have
\begin{align}\label{I3-2}
\int_{\widehat{\Omega}_{s}(z_{1})}(\mathbb{C}^{0}e(w),e(w\eta^{2}))dx
\geq\frac{1}{C}\int_{\widehat{\Omega}_{s}(z_{1})}|\nabla(w\eta)|^{2}dx
-C\int_{\widehat{\Omega}_{s}(z_{1})}|w|^{2}|\nabla\eta|^{2}dx,
\end{align}
and
$$
\bigg|\int_{\widehat{\Omega}_{s}(z_{1})}(w\eta^{2})\mathcal{L}_{\lambda,\mu}\bar{u}_{1}^{1}\,dx\bigg|
\leq\,\frac{C}{(s-t)^{2}}\int_{\widehat{\Omega}_{s}(z_{1})}|w|^{2}dx
+(s-t)^{2}\int_{\widehat{\Omega}_{s}(z_{1})}\left|\mathcal{L}_{\lambda,\mu}\bar{u}_{1}^{1}\right|^{2}dx.
$$
It follows that
\begin{align}\label{FsFt11}
\int_{\widehat{\Omega}_{t}(z_{1})}|\nabla{w}|^{2}dx\leq\,\frac{C}{(s-t)^{2}}\int_{\widehat{\Omega}_{s}(z_{1})}|w|^{2}dx
+(s-t)^{2}\int_{\widehat{\Omega}_{s}(z_{1})}\left|\mathcal{L}_{\lambda,\mu}\bar{u}_{1}^{1}\right|^{2}dx.
\end{align}

\noindent
{\bf Case 1. For $\sqrt{\epsilon}\leq|z_{1}|\leq\,R$.}

Note that for $0<s<\frac{2|z_{1}|}{3}$, we have
\begin{align}\label{energy_w_square}
\int_{\widehat{\Omega}_{s}(z_{1})}|w|^{2}dx
&=\int_{|x_{1}-z_{1}|\leq\,s}\int_{-\frac{\epsilon}{2}+h_{2}(x_{1})}^{\frac{\epsilon}{2}+h_{1}(x_{1})}
|w(x_{1},x_{2})|^{2}dx_{2}dx_{1}\nonumber\\
&\leq\int_{|x_{1}-z_{1}|\leq\,s}(\epsilon+h_{1}(x_{1})
-h_{2}(x_{1}))^{2}\int_{-\frac{\epsilon}{2}+h_{2}(x_{1})}^{\frac{\epsilon}{2}+h_{1}(x_{1})}
\left|\partial_{x_{2}}w(x_{1},x_{2})\right|^{2}dx_{2}dx_{1}\nonumber\\
&\leq\,C|z_{1}|^{4}\int_{\widehat{\Omega}_{s}(z_{1})}|\nabla{w}|^{2}dx,
\end{align}
By \eqref{L_ubar_ialpha}, we have
\begin{equation}\label{integal_Lubar11}
\int_{\widehat{\Omega}_{s}(z_{1})}\left|\mathcal{L}_{\lambda,\mu}\bar{u}_{1}^{1}\right|^{2}
dx\leq\frac{Cs}{|z_{1}|^{4}},\quad\,0<s<\frac{2|z_{1}|}{3}.
\end{equation}
Denote
$$\widehat{F}(t):=\int_{\widehat{\Omega}_{t}(z_{1})}|\nabla{w}|^{2}dx.$$
It follows from the above that
\begin{equation}\label{tildeF111}
\widehat{F}(t)\leq\,\left(\frac{C_{0}|z_{1}|^{2}}{s-t}\right)^{2}\widehat{F}(s)+C(s-t)^{2}\frac{s}{|z_{1}|^{4}},
\quad\forall~0<t<s<\frac{2|z_{1}|}{3},
\end{equation}
where $C_0$ is also a universal constant.

Let $t_{i}=2C_{0}i\,|z_{1}|^{2}$, $i=1,2,\cdots$. Then
$$\frac{C_{0}|z_{1}|^{2}}{t_{i+1}-t_{i}}=\frac{1}{2}.$$
Let $k=\left[\frac{1}{4C_{0}|z_{1}|}\right]$. Then by \eqref{tildeF111} with $s=t_{i+1}$ and $t=t_{i}$, we have
$$\widehat{F}(t_{i})\leq\,\frac{1}{4}\widehat{F}(t_{i+1})+\frac{C(t_{i+1}-t_{i})^{2}t_{i+1}}{|z_{1}|^{4}}
\leq\,\frac{1}{4}\widehat{F}(t_{i+1})+C(i+1)|z_{1}|^{2},
$$
After $k$ iterations, we have,
using (\ref{energy_w}),
\begin{eqnarray*}
\widehat{F}(t_{1}) &\leq &\left(\frac{1}{4}\right)^{k}\widehat{F}(t_{k+1})+C|z_{1}|^{2}\sum_{l=1}^{k}\left(\frac{1}{4}\right)^{l-1}(l+1)
\leq\,
C\left(\frac{1}{4}\right)^{k}
+C|z_{1}|^{2}\sum_{l=1}^{k}\left(\frac{1}{4}\right)^{l-1}(l+1)\\
&\leq &C|z_{1}|^{2}.
\end{eqnarray*}
This implies that
$$\int_{\widehat{\Omega}_{\delta}(z_{1})}|\nabla{w}|^{2}dx\leq\,C|z_{1}|^{2}.$$

\noindent
{\bf Case 2.} For $|z_{1}|\leq \sqrt{\epsilon}$.

 For $0<t<s<\sqrt{\epsilon}$, we still have \eqref{FsFt11}. Estimate \eqref{energy_w_square} becomes
 \begin{align}\label{energy_w_square_in}
\int_{\widehat{\Omega}_{s}(z_{1})}|w|^{2}dx
\leq\,C\epsilon^{2}\int_{\widehat{\Omega}_{s}(z_{1})}|\nabla{w}|^{2}dx,
\quad\,0<s<\sqrt{\epsilon}.
\end{align}
Estimate \eqref{integal_Lubar11} becomes
\begin{equation}\label{integal_Lubar11_in}
\int_{\widehat{\Omega}_{s}(z_{1})}
\left|\mathcal{L}_{\lambda,\mu}\bar{u}_{1}^{1}\right|^{2} dx
\leq\frac{Cs}{\epsilon}+\frac{C|z_{1}|^{2}s}{\epsilon^{3}}
+\frac {Cs^3}{ \epsilon^3},\quad\,0<s<\sqrt{\epsilon}.
\end{equation}
Estimate \eqref{tildeF111} becomes, in view of \eqref{FsFt11},
\begin{equation}\label{tildeF111_in}
\widehat{F}(t)\leq\,\left(\frac{C_{0}\epsilon}{s-t}\right)^{2}\widehat{F}(s)+C(s-t)^{2}s\left(\frac{1}{\epsilon}+\frac{|z_{1}|^{2}}{\epsilon^{3}}
+\frac {s^2}{ \epsilon^3}\right),
\quad\forall~0<t<s<\sqrt{\epsilon}.
\end{equation}
Let $t_{i}=2C_{0}i\epsilon$, $i=1,2,\cdots$. Then
$$\frac{C_{0}\epsilon}{t_{i+1}-t_{i}}=\frac{1}{2}.$$
Let $k=\left[\frac{1}{4C_{0}\sqrt{\epsilon}}\right]$. Then by \eqref{tildeF111_in} with $s=t_{i+1}$ and $t=t_{i}$, we have
$$\widehat{F}(t_{i})\leq\,\frac{1}{4}\widehat{F}(t_{i+1})
+Ci^3(\epsilon^{2}+|z_{1}|^{2}).$$
After $k$ iterations, we have,
using (\ref{energy_w}),
\begin{eqnarray*}
\widehat{F}(t_{1})
&\leq& \left(\frac{1}{4}\right)^{k}\widehat{F}(t_{k+1})
+C\sum_{l=1}^{k}\left(\frac{1}{4}\right)^{l-1}l^3(\epsilon^{2}+|z_{1}|^{2})
\leq\,
C\left(\frac{1}{4}\right)^{\frac{1}{C\sqrt{\epsilon}}}
+C(\epsilon^{2}+|z_{1}|^{2})\\
&\leq&\,C(\epsilon^{2}+|z_{1}|^{2}).
\end{eqnarray*}
This implies that
$$\int_{\widehat{\Omega}_{\delta}(z_{1})}|\nabla{w}|^{2}dx\leq\,C(\epsilon^{2}+|z_{1}|^{2}).$$

{\bf STEP 3.} Proof of \eqref{nabla_w_ialpha}.

Making a change of variables
\begin{equation}\label{changeofvariant}
 \left\{
  \begin{aligned}
  &x_1-z_{1}=\delta y_1,\\
  &x_2=\delta y_2,
  \end{aligned}
 \right.
\end{equation}
then $\widehat{\Omega}_{\delta}(z_{1})$ becomes $Q'_{1}$, where
$$Q'_{r}=\left\{y\in\mathbb{R}^{2}\bigg|-\frac{\epsilon}{2\delta}+\frac{1}{\delta}h_{2}(\delta{y}_{1}+z_{1})<y_{2}
<\frac{\epsilon}{2\delta}+\frac{1}{\delta}h_{1}(\delta{y}_{1}+z_{1}),|y_{1}|<r\right\},\quad\mbox{for}~~r\leq1,$$ and the boundaries $\Gamma^{\pm}_{1}$
become
$$
y_2=\hat{h}_{1}(y_{1}):=\frac{1}{\delta}
\left(\frac{\epsilon}{2}+h_{1}(\delta\,y_{1}+z_{1})\right),\quad|y_{1}|<1,$$
and
$$y_2=\hat{h}_{2}(y_{1}):=\frac{1}{\delta}\left(-\frac{\epsilon}{2}
+h_{2}(\delta\,y_{1}+z_{1})\right), \quad |y_1|<1.
$$
 Then
$$\hat{h}_{1}(0)-\hat{h}_{2}(0):=\frac{1}{\delta}\left(\epsilon+h_{1}(z_{1})-h_{2}(z_{1})\right)=2,$$
and by \eqref{h1h20} and \eqref{h1h22},
$$|\hat{h}_{1}'(0)|+|\hat{h}_{2}'(0)|\leq\,C|z_{1}|,
\quad|\hat{h}_{1}''(0)|+|\hat{h}_{2}''(0)|\leq\,C\delta.$$
Since $R$ is small, $\|\hat{h}_{1}\|_{C^{1,1}((-1,1))}$ and $\|\hat{h}_{2}\|_{C^{1,1}((-1,1))}$ are small and $\frac{1}{2}Q'_{1}$ is essentially a unit square as far as applications of Sobolev embedding
theorems and classical $L^{p}$ estimates for elliptic systems are concerned.
Let
\begin{equation}\label{def_U}
U_{1}^{1}(y_1, y_2):=\bar{u}_{1}^{1}(x_{1},x_{2}),\quad\,W_{1}^{1}(y_1, y_2):=w_{1}^{1}(x_{1},x_{2}),
\quad\,y\in{Q}'_{1},
\end{equation}
then by \eqref{w20},
\begin{align}
\mathcal{L}_{\lambda,\mu}W_{1}^{1}
=-\mathcal{L}_{\lambda,\mu}U_{1}^{1},
\quad\quad\,y\in{Q'_{1}}.
\end{align}
where
$$\left|\mathcal{L}_{\lambda,\mu}U_{1}^{1}\right|=\delta^{2}\left|\mathcal{L}_{\lambda,\mu}\bar{u}_{1}^{1}\right|.$$
Since $W_1^1=0$ on the top and
bottom boundaries of $Q'_{1}$, we have, using
 Poincar\'{e} inequality,
that
$$\left\|W_{1}^{1}\right\|_{H^{1}(Q'_{1})}\leq\,C\left\|\nabla{W}_{1}^{1}\right\|_{L^{2}(Q'_{1})}.$$
By
 $W^{2,p}$ estimates for elliptic systems
(see \cite{adn}) and
Sobolev embedding theorems,
we have, with $p=3$,
\begin{align*}
\left\|\nabla{W}_{1}^{1}\right\|_{L^{\infty}(Q'_{1/2})}\leq\,
C\left\|W_{1}^{1}\right\|_{W^{2,p}(Q'_{1/2})}
\leq\,C\left(\left\|\nabla{W}_{1}^{1}\right\|_{L^{2}(Q'_{1})}+\left\|\mathcal{L}_{\lambda,\mu}U_{1}^{1}\right\|_{L^{\infty}(Q'_{1})}\right).
\end{align*}
It follows that
\begin{equation}
\left\|\nabla{w}_{1}^{1}\right\|_{L^{\infty}(\widehat{\Omega}_{\frac{\delta}{2}}(z_{1}))}\leq\,
\frac{C}{\delta}\left(\left\|\nabla{w}_{1}^{1}\right\|_{L^{2}(\widehat{\Omega}_{\delta}(z_{1}))}
+\delta^{2}\left\|\mathcal{L}_{\lambda,\mu}\bar{u}_{1}^{1}\right\|_{L^{\infty}(\widehat{\Omega}_{\delta}(z_{1}))}\right).
\label{AAA}
\end{equation}

\noindent
{\bf Case 1.}\   For $\sqrt{\epsilon}\leq|z_{1}|\leq\,R$.

By (\ref{energy_w_inomega_z1}),
$$\int_{\widehat{\Omega}_{\delta}(z_{1})}\left|\nabla{w}_{1}^{1}\right|^{2}dx\leq\,C|z_{1}|^{2}.$$
By \eqref{L_ubar_ialpha},
$$\delta^{2}\left|\mathcal{L}_{\lambda,\mu}\bar{u}_{1}^{1}\right|\leq\delta^{2}\left(\frac{C}{|z_{1}|^{2}}+\frac{C}{|z_{1}|^{3}}\right)\leq\,C|z_{1}|,
\qquad \mbox{in}\ \widehat\Omega_\delta(z_1).$$
We deduce from (\ref{AAA}) that
$$\left|\nabla{w}_{1}^{1}(z_{1},x_{2})\right|=\frac{C|z_{1}|}{\delta}\leq\frac{C}{|z_{1}|},
\qquad\forall \
-\frac \epsilon 2 +h_2(z_1)<
x_2< \frac \epsilon 2 +h_1(z_1).$$

\noindent
{\bf Case 2.}\  For $|z_{1}|\leq\sqrt{\epsilon}$.

By (\ref{energy_w_inomega_z1}),
$$\int_{\widehat{\Omega}_{\delta}(z_{1})}\left|\nabla{w}_{1}^{1}\right|^{2}dx\leq\,C(\epsilon^{2}+|z_{1}|^{2}).$$
By \eqref{L_ubar_ialpha},
$$\delta^{2}\left|\mathcal{L}_{\lambda,\mu}\bar{u}_{1}^{1}\right|\leq
C \delta^{2}\left(\frac{1}{\epsilon}
+\frac{\epsilon+|z_1|}{\epsilon^2}\right)\leq\,C
(\epsilon+|z_1|), \qquad
\mbox{in}\
\widehat \Omega_\delta(z_1).$$
We deduce from (\ref{AAA}) that
$$\left|\nabla{w}_{1}^{1}(z_{1},x_{2})\right|=\frac{C}{\delta}
\left(\epsilon+|z_{1}|\right)\leq
C \frac{\epsilon +|z_1|}{\epsilon},
\qquad\forall \
-\frac \epsilon 2 +h_2(z_1)<
x_2< \frac \epsilon 2 +h_1(z_1).$$
Proposition \ref{prop1} is established.
\end{proof}

\subsection{Estimates of $v_{i}^{3}$, $i=1,2$}

Define
\begin{equation}\label{def:ubar1323}
\bar{u}_{1}^{3}=\left(x_{2}\bar{u},-x_{1}\bar{u}\right)^{T},\quad\mbox{and}
\quad\,\bar{u}_{2}^{3}=\left(x_{2}\underline{u},-x_{1}\underline{u}\right)^{T}
\end{equation}
then $v_{i}^{3}=\bar{u}_{i}^{3}$ on $\partial\widetilde{\Omega}$, $i=1,2$.
Using  \eqref{nablau_bar},
\eqref{h1h20}
and \eqref{h1h2}, we obtain
\begin{equation}\label{nabla_baru13_in}
\left|\nabla\bar{u}_{i}^{3}(x)\right|\leq\,
\frac{C(\epsilon+|x_1|)}
{\epsilon+|x_{1}|^{2}},\quad\,i=1,2,\quad\,x\in\Omega_{R},
\end{equation}
and
\begin{equation}\label{nabla_baru13_out}
\left|\nabla\bar{u}_{i}^{3}(x)\right|\leq\,C,\quad\,i=1,2,\quad\,x\in\widetilde{\Omega}\setminus\Omega_{R}.
\end{equation}
It follows from \eqref{def:ubar1323},
\eqref{L_u}, \eqref{nablau_bar} and \eqref{nabla2u_bar} that
\begin{equation}\label{L_ubar_i3}
\left|\mathcal{L}_{\lambda,\mu}\bar{u}_{i}^{3}\right|\leq\frac{C}{\epsilon+|x_{1}|^{2}},\quad i=1,2,
\quad\,x\in\Omega_{R}.
\end{equation}

We estimate the energy of $v_{i}^{3}$, $i=1,2$.

\begin{lemma}\label{lemvi3}
\begin{equation}\label{lem31.02}
\int_{\widetilde{\Omega}}\left|{v}_{i}^{3}\right|^{2}dx+
\int_{\widetilde{\Omega}}
\left|\nabla{v}_{i}^{3}\right|^{2}dx\leq\,C,\quad~i=1,2,
\end{equation}
and
\begin{equation}\label{nablavi3_outomega1/2}
\left\|\nabla{v}_{i}^{3}\right\|_{L^{\infty}(\widetilde{\Omega}\setminus\Omega_{R})}\leq\,C,\quad\,i=1,2.
\end{equation}
\end{lemma}

\begin{proof}
By \eqref{nabla_baru13_in} and \eqref{nabla_baru13_out}, we have
$$I_{\infty}[v_{i}^{3}]\leq\,I_{\infty}[\bar{u}_{i}^{3}]\leq\,
C\left\|\nabla\bar{u}_{i}^{3}\right\|_{L^{2}(\widetilde{\Omega})}^2\leq\,C,$$
and,  by (\ref{coeff4_strongelyconvex})
and (\ref{coeff2}) and  the First Korn inequality,
\begin{align*}
\left\|\nabla{v}_{i}^{3}\right\|_{L^{2}(\widetilde{\Omega})}&
\leq\left\|\nabla(v_{i}^{3}-\bar{u}_{i}^{3})\right\|_{L^{2}(\widetilde{\Omega})}+\left\|\nabla\bar{u}_{i}^{3}\right\|_{L^{2}(\widetilde{\Omega})}
\leq\sqrt{2}\left\|e(v_{i}^{3}-\bar{u}_{i}^{3})\right\|_{L^{2}(\widetilde{\Omega})}+C\\
&\leq\,C\left\|e(v_{i}^{3})\right\|_{L^{2}(\widetilde{\Omega})}+C\leq\,CI_{\infty}[v_{i}^{3}]+C
\leq\,C.
\end{align*}
We know from the Poincar\'{e} inequality that
$$\int_{\widetilde{\Omega}}\left|{v}_{i}^{3}\right|^{2}dx\leq\,C
\int_{\widetilde{\Omega}}\left|\nabla{v}_{i}^{3}\right|^{2}dx\leq\,C.$$
Note that the above constant $C$ is independent of $\epsilon$.

With \eqref{lem31.02},
 we can apply classical elliptic estimates, see
\cite{adn1} and \cite{adn}, to obtain
 \eqref{nablavi3_outomega1/2}.
\end{proof}

Denote
$$w_{i}^{3}:={v}_{i}^{3}-\bar{u}_{i}^{3},\quad\,i=1,2.$$
It is easy to see from \eqref{nabla_baru13_in}, \eqref{nabla_baru13_out}
and \eqref{lem31.02} that
\begin{equation} \label{XYZ}
\int_{\widetilde \Omega}
\left|\nabla w_i^3\right|^2\le C.
\end{equation}

\begin{lemma}\label{lem_I11}
With $\delta=\delta(z_{1})$ in \eqref{delta}, we have,
 for $i=1,2$,
\begin{equation}\label{energyw13_out}
\int_{\widehat{\Omega}_{\delta}(z_{1})}\left|\nabla{w}_{i}^{3}\right|^{2}dx\leq
\begin{cases}
C\epsilon^{2},& |z_{1}|<\sqrt{\epsilon},\\
C|z_{1}|^{4},&\sqrt{\epsilon}\leq|z_{1}|<R.
\end{cases}
\end{equation}
\end{lemma}

\begin{proof}
The proof is similar to that of \eqref{energy_w_inomega_z1}.
We will only prove it for $i=1$, since the proof for $i=2$ is the same. For simplicity, denote $w:=w_{1}^{3}$, then
\begin{equation}\label{w2013}
\begin{cases}
\mathcal{L}_{\lambda,\mu}w=-\mathcal{L}_{\lambda,\mu}\bar{u}_{1}^{3},&\mbox{in}~\widetilde{\Omega},\\
w=0,&\mbox{on}~\partial\widetilde{\Omega}.
\end{cases}
\end{equation}
As in the proof of \eqref{energy_w_inomega_z1}, we have, instead of \eqref{FsFt11},
\begin{align}\label{FsFt13}
\int_{\widehat{\Omega}_{t}(z_{1})}|\nabla{w}|^{2}dx\leq\,\frac{C}{(s-t)^{2}}\int_{\widehat{\Omega}_{s}(z_{1})}|w|^{2}dx
+(s-t)^{2}\int_{\widehat{\Omega}_{s}(z_{1})}\left|\mathcal{L}_{\lambda,\mu}\bar{u}_{1}^{3}\right|^{2}dx.
\end{align}

\noindent
{\bf Case 1. $\sqrt{\epsilon}<|z_{1}|<R$.}

We still have \eqref{energy_w_square} for $0<s<\frac{2|z_{1}|}{3}$. Instead of \eqref{integal_Lubar11}, we have, using
 \eqref{L_ubar_i3},
\begin{align}\label{Lu13}
\int_{\widehat{\Omega}_{s}(z_{1})}\left|\mathcal{L}_{\lambda,\mu}\bar{u}_{1}^{3}\right|^{2}dx\leq\frac{Cs}{|z_{1}|^{2}}.
\end{align}
Instead of \eqref{tildeF111}, we have
\begin{equation}\label{tildeF13}
\widehat{F}(t)\leq\,\left(\frac{C_{0}|z_{1}|^{2}}{s-t}\right)^{2}\widehat{F}(s)+C(s-t)^{2}\frac{s}{|z_{1}|^{2}},
\quad\forall~0<t<s<\frac{2|z_{1}|}{3}.
\end{equation}
We define $\{t_{i}\}$, $k$ and iterate as in the proof of \eqref{energy_w_inomega_z1}, right below formula \eqref{tildeF111}, to obtain,
using \eqref{XYZ},
$$\widehat{F}(t_{1})
\leq\,\left(\frac{1}{4}\right)^{k}\widehat{F}(\frac{2|z_{1}|}{3})+C|z_{1}|^{4}\sum_{l=1}^{k}\left(\frac{1}{4}\right)^{l}l\leq\,C|z_{1}|^{4}.$$
This implies that
$$\int_{\widehat{\Omega}_{\delta}(z_{1})}|\nabla{w}|^{2}dx\leq\,C|z_{1}|^{4}.$$

\noindent
{\bf Case 2. $|z_{1}|<\sqrt{\epsilon}$.}

Estimate \eqref{energy_w_square_in} remains the same. Estimate \eqref{integal_Lubar11_in} becomes
\begin{align}\label{integal_Lubar13_in}
\int_{\widehat{\Omega}_{s}(z_{1})}\left|\mathcal{L}_{\lambda,\mu}\bar{u}_{1}^{3}\right|^{2}dx
\leq\frac{Cs}{\epsilon},\quad\,0<s<\sqrt{\epsilon}.
\end{align}
Estimate \eqref{tildeF111_in} becomes
\begin{equation}\label{tildeF2}
\widehat{F}(t)\leq\,\left(\frac{C_{0}\epsilon}{s-t}\right)^{2}\widehat{F}(s)+\frac{C(s-t)^{2}s}{\epsilon},
\quad\forall~0<t<s<\sqrt{\epsilon}.
\end{equation}
Define $\{t_{i}\}$, $k$ and iterate as in the proof of \eqref{energy_w_inomega_z1}, right below formula \eqref{tildeF111_in}, to obtain
$$\widehat{F}(t_{1})\leq\,\left(\frac{1}{4}\right)^{k}\widehat{F}(t_{k+1})
+C\sum_{l=1}^{k}\left(\frac{1}{4}\right)^{l-1}l\epsilon^{2}
\leq\,C\epsilon^{2}.$$
This implies that
$$\int_{\widehat{\Omega}_{\delta}(z_{1})}|\nabla{w}|^{2}dx\leq\,C\epsilon^{2}.$$
\end{proof}

\begin{lemma}\label{lemma3.6}
\begin{equation}\label{nabla_w13}
\left\|\nabla{w}_{i}^{3}\right\|_{L^{\infty}(\widetilde{\Omega})}\leq\,C,\quad\,i=1,2.
\end{equation}
Consequently,
\begin{equation}\label{nabla_v13}
\left|\nabla{v}_{i}^{3}(x)\right|\leq\,\frac{C(\epsilon+|x_{1}|)}{\epsilon+|x_{1}|^{2}},\quad\,i=1,2,\quad\,x\in{\Omega_R}.
\end{equation}
\end{lemma}

\begin{proof}
The proof is the same as that of \eqref{nabla_w_ialpha}.
In Case 1, $\sqrt{\epsilon}\leq|z_{1}|\leq\,R$, we use estimates
$$\int_{\widehat{\Omega}_{\delta}(z_{1})}\left|\nabla{w}_{1}^{3}\right|^{2}dx\leq\,C|z_{1}|^{4},$$
and
$$\delta^{2}\left|\mathcal{L}_{\lambda,\mu}\bar{u}_{1}^{3}\right|\leq\,C|z_{1}|^{2}.$$

In Case 2, $|z_{1}|\leq\sqrt{\epsilon}$. we use
$$\int_{\widehat{\Omega}_{\delta}(z_{1})}\left|\nabla{w}_{1}^{3}\right|^{2}dx\leq\,C\epsilon^{2},$$
and
$$\delta^{2}\left|\mathcal{L}_{\lambda,\mu}\bar{u}_{1}^{3}\right|\leq\,C\epsilon.$$
\end{proof}

\section{Estimates of $C_{1}^{\alpha}$ and $C_{2}^{\alpha}$}\label{sec_C1C2}

In this section, we first prove that $C_{1}^{\alpha}$ and $C_{2}^{\alpha}$ are uniformly bounded with respect to $\epsilon$, and then estimate the difference  $C_{1}^{\alpha}-C_{2}^{\alpha}$.

\subsection{Boundedness of $C_{i}^{\alpha}$,  $i=1,2, \alpha=1,2,3$}

\begin{lemma}\label{lem_C1C2bound}
Let $C_{i}^{\alpha}$ be defined in \eqref{decom_u}. Then
\begin{equation*}
|C_{i}^{\alpha}|\leq\,C, \qquad
i=1,2; ~~\alpha=1,2,3.
\end{equation*}
\end{lemma}

\begin{proof}  We only need to prove it for $i=1$,
since the proof  for $i=2$ is the same.
Let $u_{\epsilon}$ be the solution of \eqref{mainequation}.
 By Theorem \ref{theorem5.6} and
Theorem \ref{thm_minimizer_infity} in the Appendix,
$$I_{\infty}[u_{\epsilon}]:=\frac{1}{2}\int_{\widetilde{\Omega}}
\left(\mathbb{C}^{(0)}
e(u_{\epsilon}),e(u_{\epsilon})\right)
\leq I_\infty[\Phi]\leq\,C$$
where $\Phi$ is the one in the proof of Lemma \ref{lem_v3_v1+v2}.

It follows that
$$\|u_{\epsilon}\|_{H^{1}(\widetilde{\Omega})}
\leq\,C\|e(u_{\epsilon})\|_{L^{2}(\widetilde{\Omega})}
\leq C I_\infty[u_\epsilon]\leq\,C.$$
By the trace embedding theorem,
$$\|u_{\epsilon}\|_{L^{2}(\partial{D}_{1}\setminus{B}_{R})}\leq\,C.$$
On $\partial{D}_{1}$,
$$u_{\epsilon}=\sum_{\alpha=1}^{3}C_{1}^{\alpha}\psi^{\alpha}.$$
If $C_{1}:=(C_{1}^{1},C_{1}^{2},C_{1}^{3})=0$, there is nothing to prove.
Otherwise
\begin{equation}\label{b3-1}
C\geq|C_{1}|
\left\|\sum_{\alpha=1}^{3}
\widehat{C}_{1}^{\alpha}\psi^{\alpha}\right\|_{L^{2}(\partial{D}_{1}\setminus{B}_{R})},
\end{equation}
where $\widehat{C}_{1}^{\alpha}=\frac{C_{1}^{\alpha}}{|C_{1}|}$ and $|\widehat{C}_{1}|=1$. It is easy to see that
\begin{equation}\label{b3-2}
\left\|\sum_{\alpha=1}^{3}\widehat{C}_{1}^{\alpha}\psi^{\alpha}\right\|_{L^{2}(\partial{D}_{1}\setminus{B}_{R})}\geq\frac{1}{C}.
\end{equation}
Indeed, if not, along a subsequence $\epsilon\rightarrow0$, $\widehat{C}_{1}^{\alpha}\rightarrow\bar{C}_{1}^{\alpha}$, and
$$\left\|\sum_{\alpha=1}^{3}\bar{C}_{1}^{\alpha}\psi^{\alpha}\right\|_{L^{2}(\partial{D}^{*}_{1}\setminus{B}_{R})}=0,$$
where $\partial{D}^{*}_{1}$ is the limit of $\partial{D}_{1}$ as $\epsilon\rightarrow0$ and $|\bar{C}_{1}|=1$. This implies $\sum_{\alpha=1}^{3}\bar{C}_{1}^{\alpha}\psi^{\alpha}=0$ on $\partial{D}^{*}_{1}\setminus{B}_{R}$. But $\left\{\psi^{\alpha}\big|_{\partial{D}^{*}_{1}\setminus{B}_{R}}\right\}$ is easily seen to be linear independent, we must have $\bar{C}_{1}=0$. This is a contradiction.
 Lemma \ref{lem_C1C2bound} for $i=1$
 follows from \eqref{b3-1} and \eqref{b3-2}.
\end{proof}

\subsection{Estimates of $|C_{1}^{\alpha}-C_{2}^{\alpha}|$, $\alpha=1,2$}

In the rest of this section, we prove
\begin{prop}\label{prop_C1-C2}
Let $C_{i}^{\alpha}$ be defined in \eqref{decom_u}. Then
\begin{equation*}
|C_{1}^{\alpha}-C_{2}^{\alpha}|\leq\,C\sqrt{\epsilon},\quad\alpha=1,2.
\end{equation*}
\end{prop}

By  the fourth line of \eqref{mainequation},
\begin{align}\label{C1C2_2}
\sum_{\alpha=1}^{3}C_{1}^{\alpha}\int_{\partial{D}_{j}}\frac{\partial{v}_{1}^{\alpha}}{\partial\nu_{0}}\bigg|_{+}\cdot\psi^{\beta}
&+\sum_{\alpha=1}^{3}C_{2}^{\alpha}\int_{\partial{D}_{j}}\frac{\partial{v}_{2}^{\alpha}}{\partial\nu_{0}}\bigg|_{+}\cdot\psi^{\beta}
+\int_{\partial{D}_{j}}\frac{\partial{v}_{3}}{\partial\nu_{0}}\bigg|_{+}\cdot\psi^{\beta}=0,\nonumber\\
&\quad\,j=1,2;~~\beta=1,2,3.
\end{align}

Denote
\begin{align*}
a_{ij}^{\alpha\beta}=-\int_{\partial{D}_{j}}\frac{\partial{v}_{i}^{\alpha}}{\partial\nu_{0}}\bigg|_{+}\cdot\psi^{\beta},\quad
b_{j}^{\beta}=\int_{\partial{D}_{j}}\frac{\partial{v}_{3}}{\partial\nu_{0}}\bigg|_{+}\cdot\psi^{\beta},\quad\,i,j=1,2;~\alpha,\beta=1,2,3.
\end{align*}
Integrating by parts
over $\widetilde \Omega$ and using
(\ref{v1alpha}), we have
$$a_{ij}^{\alpha\beta}=\int_{\widetilde{\Omega}}\left(\mathbb{C}^{0}e(v_{i}^{\alpha}),e(v_{j}^{\beta})\right)dx,\quad
b_{j}^{\beta}=-\int_{\widetilde{\Omega}}\left(\mathbb{C}^{0}e(v_{3}),e(v_{j}^{\beta})\right)dx.$$
Then \eqref{C1C2_2} can be written as
\begin{equation}\label{C1C2_3}
\left\{
\begin{aligned}
\sum_{\alpha=1}^{3}C_{1}^{\alpha}a_{11}^{\alpha\beta}+\sum_{\alpha=1}^{3}C_{2}^{\alpha}a_{21}^{\alpha\beta}-b_{1}^{\beta}&=0,\\
\sum_{\alpha=1}^{3}C_{1}^{\alpha}a_{12}^{\alpha\beta}+\sum_{\alpha=1}^{3}C_{2}^{\alpha}a_{22}^{\alpha\beta}-b_{2}^{\beta}&=0,
\end{aligned}
\right.\quad\quad~~\beta=1,2,3.
\end{equation}
For simplicity, we use $a_{ij}$ to denote the $3\times3$ matrix $(a_{ij}^{\alpha\beta})$. To estimate $|C_{1}^{\alpha}-C_{2}^{\alpha}|$, $\alpha=1,2$, we only need to use the first three equations in \eqref{C1C2_3}:
$$ a_{11}C_{1} +a_{21}C_{2} =b_{1},$$
where
$$C_{1}=(C_{1}^{1}, C_{1}^{2}, C_{1}^{3})^{T} ,
\quad\,C_{2}=(C_{2}^{1}, C_{2}^{2}, C_{2}^{3})^{T},
\quad\,b_{1}=(b_{1}^{1},b_{1}^{2},b_{1}^{3})^{T}.$$
We  write the equation as
\begin{equation}\label{C1C2_4}
a_{11}(C_{1}-C_{2})=p:=b_{1}-(a_{11}+a_{21})C_{2}.
\end{equation}
Namely,
\begin{equation}\label{AX=p}
a_{11}(C_{1}-C_{2}) \equiv
\begin{pmatrix}
  a_{11}^{11} & a_{11}^{12} & a_{11}^{13} \\\\
  a_{11}^{21} & a_{11}^{22} & a_{11}^{23} \\\\
 a_{11}^{31} & a_{11}^{32} & a_{11}^{33} \\
\end{pmatrix}\begin{pmatrix}
               C_{1}^{1}-C_{2}^{1} \\\\
               C_{1}^{2}-C_{2}^{2} \\\\
               C_{1}^{3}-C_{2}^{3} \\
             \end{pmatrix}=\begin{pmatrix}
               p^{1}\\
               p^{2} \\
               p^{3} \\
             \end{pmatrix}.
\end{equation}

We will show that $a_{11}$ is positive definite, which we assume for the
 time being.
By Cramer's rule, we see from \eqref{AX=p},
$$C_{1}^{1}-C_{2}^{1}=\frac{1}{\det{a_{11}}}
\begin{vmatrix}
  p^{1} & a_{11}^{12} & a_{11}^{13} \\\\
  p^{2} & a_{11}^{22} & a_{11}^{23} \\\\
 p^{3} & a_{11}^{32} & a_{11}^{33} \\
\end{vmatrix},\quad
C_{1}^{2}-C_{2}^{2}=\frac{1}{\det{a_{11}}}
\begin{vmatrix}
  a_{11}^{11} & p^{1} & a_{11}^{13} \\\\
  a_{11}^{21} & p^{2} & a_{11}^{23} \\\\
 a_{11}^{31} & p^{3} & a_{11}^{33} \\
\end{vmatrix}.$$
Therefore
\begin{align}\label{aaa1}
C_{1}^{1}-C_{2}^{1}=\frac{1}{\det{a_{11}}}\left(
              p^{1}\begin{vmatrix}
                     a_{11}^{22} & a_{11}^{23} \\\\
                     a_{11}^{32} & a_{11}^{33} \\
                   \end{vmatrix}
             -p^{2}\begin{vmatrix}
                     a_{11}^{12} & a_{11}^{13} \\\\
                     a_{11}^{32} & a_{11}^{33} \\
                   \end{vmatrix} +
               p^{3} \begin{vmatrix}
                     a_{11}^{12} & a_{11}^{13} \\\\
                     a_{11}^{22} & a_{11}^{23} \\
                   \end{vmatrix}
             \right),
\end{align}
and
\begin{align}\label{aaa2}
  C_{1}^{2}-C_{2}^{2}=\frac{1}{\det{a_{11}}}\left(
              -p^{1}\begin{vmatrix}
                     a_{11}^{21} & a_{11}^{23} \\\\
                     a_{11}^{31} & a_{11}^{33} \\
                   \end{vmatrix}
              +
               p^{2}\begin{vmatrix}
                     a_{11}^{11} & a_{11}^{13} \\\\
                     a_{11}^{31} & a_{11}^{33} \\
                   \end{vmatrix}
               -p^{3} \begin{vmatrix}
                     a_{11}^{11} & a_{11}^{13} \\\\
                     a_{11}^{21} & a_{11}^{23} \\
                   \end{vmatrix}
            \right).
\end{align}
In order to prove Proposition \ref{prop_C1-C2}, we
first study the right hand side of \eqref{AX=p} and have the following estimates.
\begin{lemma}\label{lem_ap}
\begin{align*}
&\left|a_{11}^{\alpha\beta}+a_{21}^{\alpha\beta}\right|\leq\,C,\quad\alpha,\beta=1,2,3;\\\\
&\left|b_{1}^{\beta}\right|\leq\,C,\quad\beta=1,2,3.
\end{align*}
Consequently,
\begin{equation}\label{p_bound}
|p|\leq\,C.
\end{equation}
\end{lemma}

\begin{proof}
For $\beta=1,2,3$, using \eqref{nablav1} and \eqref{lem31.02},
\begin{align}\label{nablav11_L1}
\int_{\widetilde{\Omega}}\left|\nabla{v}_{1}^{\beta}\right|dx
\leq\int_{\Omega_{R/2}}\left|\nabla{v}_{1}^{\beta}\right|dx+\int_{\widetilde{\Omega}\setminus\Omega_{R/2}}\left|\nabla{v}_{1}^{\beta}\right|dx
\leq\,C.
\end{align}
For $\alpha,\beta=1,2,3$, by
Lemma \ref{lem_v3_v1+v2} and \eqref{nablav11_L1},
 we have
\begin{align*}
\left|a_{11}^{\alpha\beta}+a_{21}^{\alpha\beta}\right|
&=\left|\int_{\widetilde{\Omega}}\left(\mathbb{C}^{0}e(v_{1}^{\alpha}+v_{2}^{\alpha}),e(v_{1}^{\beta})\right)dx\right|\\
&\leq\,C\left\|\nabla(v_{1}^{\alpha}+v_{2}^{\alpha})\right\|_{L^{\infty}(\widetilde{\Omega})}\int_{\widetilde{\Omega}}\left|\nabla{v}_{1}^{\beta}\right|dx\\
&\leq\,C.
\end{align*}
Similarly, it follows from
Lemma \ref{lem_v3_v1+v2}
and
 \eqref{nablav11_L1} that
$$\left|b_{1}^{\beta}\right|=\left|\int_{\widetilde{\Omega}}\left(\mathbb{C}^{0}e(v_{1}^{\beta}),e(v_{3})\right)dx\right|
\leq\,C\|\nabla{v}_{3}\|_{L^{\infty}(\widetilde{\Omega})}\int_{\widetilde{\Omega}}\left|\nabla{v}_{1}^{\beta}\right|dx
\leq\,C,\quad\beta=1,2,3.$$
Lemma \ref{lem_ap} follows immediately, in view of Lemma \ref{lem_C1C2bound}.
\end{proof}

\begin{lemma}\label{lem_a_11}
$a_{11}$ is positive definite, and
\begin{equation}\label{a11_1122}
\frac{1}{C\sqrt{\epsilon}}\leq\,a_{11}^{\alpha\alpha}\leq\frac{C}{\sqrt{\epsilon}},\quad\alpha=1,2,
\end{equation}
\begin{equation}\label{a11_33}
\frac{1}{C}\leq\,a_{11}^{33}\leq\,C,\quad\alpha=1,2;
\end{equation}
\begin{equation}\label{a11_12}
\left|a_{11}^{12}\right|=\left|a_{11}^{21}\right|\leq\frac{C}{\epsilon^{1/4}},
\end{equation}
\begin{equation}\label{a11_alpha3}
|a_{11}^{\alpha3}|=|a_{11}^{3\alpha}|\leq\,C,\quad\alpha=1,2;
\end{equation}
and
\begin{equation}\label{a11_det}
\frac{1}{C\epsilon}\leq\det{a_{11}}\leq\frac{C}{\epsilon}.
\end{equation}
\end{lemma}

\begin{proof}

{\bf STEP 1.} Proof of \eqref{a11_1122} and \eqref{a11_33}.

For any
$\xi=(\xi^{(1)},\xi^{(2)},\xi^{(3)})^{T}\neq0$,
$$
\xi^{T}a_{11}\xi=\int_{\widetilde{\Omega}}\left(\mathbb{C}^{0}e\left(\xi^{(\alpha)}v_{1}^{\alpha}\right),
e\left(\xi^{(\beta)}v_{1}^{\beta}\right)\right)dx
\geq\frac{1}{C}\int_{\widetilde{\Omega}}\left|e\left(\xi^{(\alpha)}v_{1}
^{\alpha}\right)\right|^{2}dx>0.
$$
In the last inequality we have used the fact that $e\left(\xi^{(\alpha)}v_{1}^{\alpha}\right)$ is not identically zero.
Indeed if $e\left(\xi^{(\alpha)}v_{1}^{\alpha}\right)=0$, then $\xi^{(\alpha)}v_{1}^{\alpha}=a\psi^{1}+b\psi^{2}+c\psi^{3}$
in $\widetilde \Omega$
for some constants $a,b$ and $c$. On the other hand, $\xi^{(\alpha)}v_{1}^{\alpha}=0$ on $\partial{D}_{2}$, and $\psi^{1}\big|_{\partial{D}_{2}}$, $\psi^{2}\big|_{\partial{D}_{2}}$ and $\psi^{3}\big|_{\partial{D}_{2}}$ are clearly independent. This implies that $a=b=c=0$. Thus on $\partial{D}_{1}$, $\xi^{(\alpha)}v_{1}^{\alpha}=0$, violating the linear independence of $\psi^{1}\big|_{\partial{D}_{1}}$, $\psi^{2}\big|_{\partial{D}_{1}}$ and $\psi^{3}\big|_{\partial{D}_{1}}$.  We have proved
 that $a_{11}$ is positive definite.

By   \eqref{coeff4_strongelyconvex},
\eqref{coeff2} and \eqref{mainev1},
$$a_{11}^{\alpha\alpha}=\int_{\widetilde{\Omega}}\left(\mathbb{C}^{0}e\left(v_{1}^{\alpha}\right),e\left(v_{1}^{\alpha}\right)\right)dx
\leq\,C\int_{\widetilde{\Omega}}\left|\nabla{v}_{1}^{\alpha}\right|^{2}dx\leq\frac{C}{\sqrt{\epsilon}},\quad\,\alpha=1,2.$$

With   \eqref{def_w}, we have, by  \eqref{energy_w},
\begin{eqnarray*}
a_{11}^{11}&=&
\int_{\widetilde{\Omega}}\left(\mathbb{C}^{0}e\left(v_{1}^{1}\right),
e\left(v_{1}^{1}\right)\right)dx
\geq\frac{1}{C}\int_{\widetilde{\Omega}}\left|e\left(v_{1}^{1}\right)
\right|^{2}dx\\
&\geq&
\frac{1}{2C}\int_{\widetilde{\Omega}}\left|e\left(\bar u_{1}^{1}\right)
\right|^{2}dx
-C \int_{\widetilde{\Omega}}\left|e\left(w_{1}^{1}\right)
\right|^{2}dx\\
&\geq&
\frac 1{2C}
\int_{\widetilde{\Omega}}\left|e\left(\bar u_{1}^{1}\right)
\right|^{2}dx
-C.
\end{eqnarray*}
Since
\begin{equation}\label{inequality}
\left|e\left(\bar u_{1}^{1}\right)
\right|^{2}
\ge \frac 14
|\partial_{x_2}\bar u|^2,
\end{equation}
we have
\begin{eqnarray*}
\int_{\widetilde{\Omega}}
\left|e\left(\bar u_{1}^{1}\right)
\right|^{2}
dx
&\geq&
\frac 14
\int_{\widetilde\Omega} |\partial_{x_2}\bar u|^2 dx
\ge
\frac 14
 \int_{\Omega_R} \frac{dx}
{  (\epsilon+h_1(x_1)-h_2(x_1))^2}\\
&\geq&
\frac 1C \int_{\Omega_R}
\frac {dx}{   (\epsilon+|x_1|^2)^2 }\ge \frac 1{  C\sqrt{\epsilon}  }.
\end{eqnarray*}
Thus
$$
a_{11}^{11}
\geq\frac{1}{C\sqrt{\epsilon}}.
$$
Similarly, we have
$$a_{11}^{22}\geq\frac{1}{C\sqrt{\epsilon}}.$$
Estimate \eqref{a11_1122} is proved.

By Lemma \ref{lemvi3},
\begin{equation}\label{a11_33_proof}
a_{11}^{33}=\int_{\widetilde{\Omega}}\left(\mathbb{C}^{0}e\left(v_{1}^{3}\right),e\left(v_{1}^{3}\right)\right)dx\leq\,C.
\end{equation}
Claim: There exists $C$ which is independent of $\epsilon$ such that for any $v\in{H}^{1}(\Omega_{R}\setminus\Omega_{R/2})$ satisfying $v=0$ on $\Gamma^{1}_{R}\setminus\Gamma^{-}_{R/2}$, it holds
\begin{equation}\label{a11_33_claim}
\left\|\nabla{v}\right\|_{L^{2}\left(\Omega_{R}\setminus\Omega_{R/2}\right)}
\leq\,C\left\|e(v)\right\|_{L^{2}\left(\Omega_{R}\setminus\Omega_{R/2}\right)}.
\end{equation}
Proof of the claim. Suppose the contrary, along a sequence of $\epsilon_{j}\rightarrow0^{+}$, there exist $\{v_{j}\}_{j=1}^{\infty}\subset{H}^{1}\left(\Omega_{R}\setminus\Omega_{R/2}\right)$ (we still omit the superscript $\epsilon_{j}$) satisfying $v_{j}=0$ on $\Gamma^{-}_{R}\setminus\Gamma^{-}_{R/2}$, and
\begin{equation}\label{a11_33_claim1}
\left\|\nabla{v}_{j}\right\|_{L^{2}\left(\Omega_{R}\setminus\Omega_{R/2}\right)}
\geq\,j\left\|e(v_{j})\right\|_{L^{2}\left(\Omega_{R}\setminus\Omega_{R/2}\right)}.
\end{equation}
By Lemma \ref{lemmaC}, we have
\begin{equation}\label{a11_33_claim2}
\left\|\nabla{v}_{j}\right\|_{L^{2}\left(\Omega_{R}\setminus\Omega_{R/2}\right)}
\leq\,C\left(\left\|e(v_{j})\right\|_{L^{2}\left(\Omega_{R}\setminus\Omega_{R/2}\right)}
+\left\|v_{j}\right\|_{L^{2}\left(\Omega_{R}\setminus\Omega_{R/2}\right)}\right),
\end{equation}
where $C$ is independent of $j$. Replacing $v_{j}$ by $\frac{v_{j}}{\|v_{j}\|_{L^{2}\left(\Omega_{R}\setminus\Omega_{R/2}\right)}}$, we may assume without loss of generality that
\begin{equation}\label{a11_33_claim3}
\|v_{j}\|_{L^{2}\left(\Omega_{R}\setminus\Omega_{R/2}\right)}=1.
\end{equation}
It follows from \eqref{a11_33_claim1}, \eqref{a11_33_claim2} and \eqref{a11_33_claim3} that
\begin{equation}\label{a11_33_claim4}
\lim_{j\rightarrow\infty}\|e(v_{j})\|_{L^{2}\left(\Omega_{R}\setminus\Omega_{R/2}\right)}=0,
\end{equation}
\begin{equation}\label{a11_33_claim5}
\|v_{j}\|_{H^{1}\left(\Omega_{R}\setminus\Omega_{R/2}\right)}\leq\,C.
\end{equation}
Let
$$\Omega_{r}^{*}:=\left\{x\in\mathbb{R}^{2}~\big|
~h_{2}(x_{1})<x_{2}<h_{1}(x_{1}),|x_{1}|<r\right\}$$
and
$$(\Gamma^{*})_{r}^{-}:=\left\{x\in\mathbb{R}^{2}~\big|
~x_{2}=h_{2}(x_{1}),|x_{1}|<r\right\}$$
denote the limits of $\Omega_{r}$ and $\Gamma_{r}^{-}$ as $\epsilon\rightarrow0$.

We can easily construct a $C^{1}$ diffeomorphism $\phi_{\epsilon}:\Omega_{R}\setminus\Omega_{R/2}\rightarrow\Omega^{*}_{R}\setminus\Omega^{*}_{R/2}$ satisfying $\phi_{\epsilon}(\Gamma_{R}^{-}\setminus\Gamma_{R/2}^{-})
=(\Gamma^{*})_{R}^{-}\setminus(\Gamma^{*})_{R/2}^{-}$ and
\begin{equation}\label{D-1}
\|\nabla\phi_{\epsilon}-I\|_{C^{0}(\Omega_{R}\setminus\Omega_{R/2})},
\|\nabla(\phi_{\epsilon})^{-1}-I\|_{C^{0}(\Omega^{*}_{R}\setminus\Omega^{*}_{R/2})}\rightarrow0,
\quad\mbox{as}~\epsilon\rightarrow0^{+},
\end{equation}
where $I$ denotes the identity matrix. Let $\widehat{v}_{j}:=v_{j}\circ(\phi_{\epsilon_{j}})^{-1}$.
We deduce from \eqref{a11_33_claim4} and \eqref{a11_33_claim5} that, along a subsequence, $\widehat{v}_{j}\rightharpoonup{v}^{*}$ weakly in $H^{1}\left(\Omega^{*}_{R}\setminus\Omega^{*}_{R/2}\right)$, where $v^{*}$ satisfies
\begin{equation}\label{a11_33_claim6}
e(v^{*})=0,\quad\mbox{in}~\Omega^{*}_{R}\setminus\Omega^{*}_{R/2},
\end{equation}
and
\begin{equation}\label{a11_33_claim7}
v^{*}=0,\quad\mbox{on}~(\Gamma^{*})_{R}^{-}\setminus(\Gamma^{*})_{R/2}^{-}.
\end{equation}
By \eqref{a11_33_claim6}, $v^{*}\in\Psi$ in $\Omega^{*}_{R}\setminus\Omega^{*}_{R/2}$. Thus, in view of \eqref{a11_33_claim7}, $v^{*}\equiv0$ in $\Omega^{*}_{R}\setminus\Omega^{*}_{R/2}$. By the compact embedding theorem of $H^{1}(\Omega^{*}_{R}\setminus\Omega^{*}_{R/2})$ to $L^{2}(\Omega^{*}_{R}\setminus\Omega^{*}_{R/2})$,
$$\|\widehat{v}_{j}\|_{L^{2}\left(\Omega^{*}_{R}\setminus\Omega^{*}_{R/2}\right)}\rightarrow
\|v^{*}\|_{L^{2}\left(\Omega^{*}_{R}\setminus\Omega^{*}_{R/2}\right)}=0,$$
By \eqref{a11_33_claim3} and \eqref{D-1},
$$\|\widehat{v}_{j}\|_{L^{2}\left(\Omega^{*}_{R}\setminus\Omega^{*}_{R/2}\right)}\rightarrow1.$$
These lead to a contradiction. The claim has been proved.

With the claim \eqref{a11_33_claim}, we obtain from \eqref{coeff4_strongelyconvex} that
$$a_{11}^{33}=\int_{\widetilde{\Omega}}\left(\mathbb{C}^{0}e\left(v_{1}^{3}\right),e\left(v_{1}^{3}\right)\right)dx
\geq\frac{1}{C}\int_{\Omega_{R}\setminus\Omega_{R/2}}|e\left(v_{1}^{3}\right)|^{2}dx
\geq\frac{1}{C}\int_{\Omega_{R}\setminus\Omega_{R/2}}|\nabla{v}_{1}^{3}|^{2}dx
\geq\frac{1}{C}.
$$
Combining with \eqref{a11_33_proof}, estimate \eqref{a11_33} is proved.

{\bf STEP 2.} Proof of \eqref{a11_12}.

Notice that
$$a_{11}^{12}=a_{11}^{21}=\int_{\widetilde{\Omega}}\left(\mathbb{C}^{0}e\left(v_{1}^{1}\right),e(v_{1}^{2})\right)dx
=\int_{\widetilde{\Omega}}\left(\mathbb{C}^{0}\nabla{v}_{1}^{1},\nabla{v}_{1}^{2}\right)dx.$$
With \eqref{def_w}, we have
\begin{align}\label{v11_v22_n2}
\int_{\Omega_{R/2}}\left(\mathbb{C}^{0}\nabla{v}_{1}^{1},\nabla{v}_{1}^{2}\right)dx
=&\int_{\Omega_{R/2}}\left(\mathbb{C}^{0}\nabla\left(\bar{u}_{1}^{1}+w_{1}^{1}\right),\nabla\left(\bar{u}_{1}^{2}+w_{1}^{2}\right)\right)dx\nonumber\\
=&\int_{\Omega_{R/2}}\left(\mathbb{C}^{0}\nabla\bar{u}_{1}^{1},\nabla\bar{u}_{1}^{2}\right)dx
+\int_{\Omega_{R/2}}\left(\mathbb{C}^{0}\nabla\bar{u}_{1}^{1},\nabla{w}_{1}^{2}\right)dx\nonumber\\
&+\int_{\Omega_{R/2}}\left(\mathbb{C}^{0}\nabla\bar{u}_{1}^{2},\nabla{w}_{1}^{1}\right)dx
+\int_{\Omega_{R/2}}\left(\mathbb{C}^{0}\nabla{w}_{1}^{1},\nabla{w}_{1}^{2}\right)dx.
\end{align}
By the definition $\bar{u}_{1}^{1}=(\bar{u},0)^{T}$ and $\bar{u}_{1}^{2}=(0,\bar{u})^{T}$, we have
\begin{equation}\label{nablau_bar1112}
\nabla\bar{u}_{1}^{1}
=\begin{pmatrix}
                      \partial_{x_{1}}\bar{u} & \partial_{x_{2}}\bar{u} \\
                      0 & 0 \\
                    \end{pmatrix},
\quad\mbox{and}\quad\nabla\bar{u}_{1}^{2}=\begin{pmatrix}
                      0 & 0 \\
                      \partial_{x_{1}}\bar{u} & \partial_{x_{2}}\bar{u} \\
                    \end{pmatrix}.
\end{equation}
By \eqref{energy_w},
$$\left|\int_{\Omega_{R/2}}\left(\mathbb{C}^{0}\nabla{w}_{1}^{1},\nabla{w}_{1}^{2}\right)dx\right|
\leq\,C\left(\int_{\Omega_{R/2}}\left|\nabla{w}_{1}^{1}\right|^{2}dx\right)^{1/2}\left(\int_{\Omega_{R/2}}\left|\nabla{w}_{1}^{2}\right|^{2}dx\right)^{1/2}\leq\,C,$$
and
\begin{equation}\label{u11w12}
\left|\int_{\Omega_{R/2}}\left(\mathbb{C}^{0}\nabla\bar{u}_{1}^{1},\nabla{w}_{1}^{2}\right)dx\right|\leq\,C
\left(\int_{\Omega_{R/2}}\left|\nabla\bar{u}_{1}^{1}\right|^{2}dx\right)^{1/2}
\left(\int_{\Omega_{R/2}}\left|\nabla{w}_{1}^{2}\right|^{2}dx\right)^{1/2}\leq\,\frac{C}{\epsilon^{1/4}}.
\end{equation}
Similarly,
\begin{equation}\label{u12w11}
\left|\int_{\Omega_{R/2}}\left(\mathbb{C}^{0}\nabla\bar{u}_{1}^{2},\nabla{w}_{1}^{1}\right)dx\right|\leq\,\frac{C}{\epsilon^{1/4}}.
\end{equation}
On the other hand,
\begin{align*}
\left(\mathbb{C}^{0}\nabla\bar{u}_{1}^{1},\nabla\bar{u}_{1}^{2}\right)
&=\begin{pmatrix}
                      (\lambda+2\mu)\partial_{x_{1}}\bar{u} & \mu\partial_{x_{2}}\bar{u} \\
                      \mu\partial_{x_{2}}\bar{u} & \lambda\partial_{x_{1}}\bar{u} \\
                    \end{pmatrix}:\begin{pmatrix}
                      0 & 0 \\
                      \partial_{x_{1}}\bar{u} & \partial_{x_{2}}\bar{u} \\
                    \end{pmatrix}
                    =(\lambda+\mu)\partial_{x_{1}}\bar{u}\partial_{x_{2}}\bar{u}.
\end{align*}
Thus,
$$\left|\int_{\Omega_{R/2}}\left(\mathbb{C}^{0}\nabla\bar{u}_{1}^{1},\nabla\bar{u}_{1}^{2}\right)dx\right|
\leq\,C\int_{\Omega_{R/2}}|\partial_{x_{1}}\bar{u}||\partial_{x_{2}}\bar{u}|dx
\leq\,C\int_{\Omega_{R/2}}\frac{|x_{1}|dx}{(\epsilon+|x_{1}|^{2})^{2}}\leq\,C|\ln\epsilon|.$$
Substituting these estimates above into \eqref{v11_v22_n2}, and using \eqref{nablav1}, we have
$$\left|a_{11}^{12}\right|=\left|a_{11}^{21}\right|
=\left|\int_{\widetilde{\Omega}}\left(\mathbb{C}^{0}\nabla{v}_{1}^{1},\nabla{v}_{1}^{2}\right)dx\right|\leq
\left|\int_{\Omega_{R/2}}\left(\mathbb{C}^{0}\nabla{v}_{1}^{1},\nabla{v}_{1}^{2}\right)dx\right|+C
\leq\frac{C}{\epsilon^{1/4}}.$$
The proof of \eqref{a11_12} is finished.

{\bf STEP 3.} Proof of \eqref{a11_alpha3}.

$$a_{11}^{\alpha3}=a_{11}^{3\alpha}=\int_{\widetilde{\Omega}}\left(\mathbb{C}^{0}e\left(v_{1}^{\alpha}\right),e(v_{1}^{3})\right)dx
=\int_{\widetilde{\Omega}}\left(\mathbb{C}^{0}\nabla{v}_{1}^{\alpha},\nabla{v}_{1}^{3}\right)dx,\quad\alpha=1,2.$$
Similarly to the above, using \eqref{energy_w} and
\eqref{XYZ}, we have, for $\alpha=1$,
\begin{align*}
a_{11}^{13}=&\int_{\Omega_{R/2}}\left(\mathbb{C}^{0}\nabla{v}_{1}^{1},\nabla{v}_{1}^{3}\right)dx+O(1)\\
=&\int_{\Omega_{R/2}}\left(\mathbb{C}^{0}\nabla\bar{u}_{1}^{1},\nabla\bar{u}_{1}^{3}\right)dx
+\int_{\Omega_{R/2}}\left(\mathbb{C}^{0}\nabla\bar{u}_{1}^{1},\nabla{w}_{1}^{3}\right)dx\nonumber\\
&+\int_{\Omega_{R/2}}\left(\mathbb{C}^{0}\nabla\bar{u}_{1}^{3},\nabla{w}_{1}^{1}\right)dx
+\int_{\Omega_{R/2}}\left(\mathbb{C}^{0}\nabla{w}_{1}^{1},\nabla{w}_{1}^{3}\right)dx+O(1)\\
=&\int_{\Omega_{R/2}}\left(\mathbb{C}^{0}\nabla\bar{u}_{1}^{1},\nabla\bar{u}_{1}^{3}\right)dx
+\int_{\Omega_{R/2}}\left(\mathbb{C}^{0}\nabla\bar{u}_{1}^{3},\nabla{w}_{1}^{1}\right)dx
+\int_{\Omega_{R/2}}\left(\mathbb{C}^{0}\nabla\bar{u}_{1}^{1},\nabla{w}_{1}^{3}\right)dx
+O(1)\\
=&:I+II+III+O(1).
\end{align*}
By the definition of $\bar{u}_{1}^{3}=(x_{2}\bar{u},-x_{1}\bar{u})^{T}$, we have
$$\nabla\bar{u}_{1}^{3}=\begin{pmatrix}
                      x_{2}\partial_{x_{1}}\bar{u} & \bar{u}+x_{2}\partial_{x_{2}}\bar{u} \\
                      -\bar{u}-x_{1}\partial_{x_{1}}\bar{u} & -x_{1}\partial_{x_{2}}\bar{u}
                    \end{pmatrix}.$$
Then
\begin{align*}
\left(\mathbb{C}^{0}\nabla\bar{u}_{1}^{1},\nabla\bar{u}_{1}^{3}\right)
&=\begin{pmatrix}
                      (\lambda+2\mu)\partial_{x_{1}}\bar{u} & \mu\partial_{x_{2}}\bar{u} \\
                      \mu\partial_{x_{2}}\bar{u} & \lambda\partial_{x_{1}}\bar{u} \\
                    \end{pmatrix}:\begin{pmatrix}
                      x_{2}\partial_{x_{1}}\bar{u} & \bar{u}+x_{2}\partial_{x_{2}}\bar{u} \\
                      -\bar{u}-x_{1}\partial_{x_{1}}\bar{u} & -x_{1}\partial_{x_{2}}\bar{u} \\
                    \end{pmatrix}\\
&=(\lambda+2\mu)x_{2}\left(\partial_{x_{1}}\bar{u}\right)^{2}+\mu\,x_{2}\left(\partial_{x_{2}}\bar{u}\right)^{2}
-(\lambda+\mu)x_{1}\partial_{x_{1}}\bar{u}\partial_{x_{2}}\bar{u}.
\end{align*}
Hence, by \eqref{nablau_bar},
\begin{align*}
&|I|=\left|\int_{\Omega_{R/2}}\left(\mathbb{C}^{0}\nabla\bar{u}_{1}^{1},\nabla\bar{u}_{1}^{3}\right)dx\right|\\
&\leq\,C\left(\int_{\Omega_{R/2}}\frac{|x_{2}||x_{1}|^{2}}{(\epsilon+|x_{1}|^{2})^{2}}\,dx+
\int_{\Omega_{R/2}}\frac{|x_{2}|}{(\epsilon+|x_{1}|^{2})^{2}}\,dx
+\int_{\Omega_{R/2}}\frac{|x_{1}|^{2}}{(\epsilon+|x_{1}|^{2})^{2}}\,dx\right)\\
&\leq\,C.
\end{align*}
By \eqref{energy_w} and \eqref{nabla_baru13_in},
\begin{align*}
|II|&=\left|\int_{\Omega_{R/2}}\left(\mathbb{C}^{0}\nabla\bar{u}_{1}^{3},\nabla{w}_{1}^{1}\right)dx\right|
\leq
C \left(\int_{\Omega_{R/2}}\left|\nabla\bar{u}_{1}^{3}\right|^{2}dx\right)^{1/2}\left(\int_{\Omega_{R/2}}\left|\nabla{w}_{1}^{1}\right|^{2}dx\right)^{1/2}
\leq\,C.
\end{align*}
While, by \eqref{nabla_w13},
\begin{align*}
|III|=
\left| \int_{\Omega_{R/2}}\left(\mathbb{C}^{0}\nabla\bar{u}_{1}^{1},
\nabla{w}_{1}^{3}\right)dx
\right|
\leq\,C\int_{\Omega_{R/2}}\left|\nabla\bar{u}_{1}^{1}\right|dx\leq\,C.
\end{align*}
Therefore
$$\left|a_{11}^{13}\right|\leq\,C.$$

Similarly, using \eqref{energy_w} and  \eqref{XYZ},
\begin{align*}
a_{11}^{23}=&\int_{\Omega_{R/2}}\left(\mathbb{C}^{0}\nabla{v}_{1}^{2},\nabla{v}_{1}^{3}\right)dx+O(1)\\
=&\int_{\Omega_{R/2}}\left(\mathbb{C}^{0}\nabla\bar{u}_{1}^{2},\nabla\bar{u}_{1}^{3}\right)dx
+\int_{\Omega_{R/2}}\left(\mathbb{C}^{0}\nabla\bar{u}_{1}^{2},\nabla{w}_{1}^{3}\right)dx\nonumber\\
&+\int_{\Omega_{R/2}}\left(\mathbb{C}^{0}\nabla\bar{u}_{1}^{3},\nabla{w}_{1}^{2}\right)dx
+\int_{\Omega_{R/2}}\left(\mathbb{C}^{0}\nabla{w}_{1}^{2},\nabla{w}_{1}^{3}\right)dx+O(1)\\
=&\int_{\Omega_{R/2}}\left(\mathbb{C}^{0}\nabla\bar{u}_{1}^{2},\nabla\bar{u}_{1}^{3}\right)dx
+\int_{\Omega_{R/2}}\left(\mathbb{C}^{0}\nabla\bar{u}_{1}^{2},\nabla{w}_{1}^{3}\right)dx+O(1).
\end{align*}
By the definition $\bar{u}_{1}^{2}$ and $\bar{u}_{1}^{3}$, we have
\begin{align*}
\left(\mathbb{C}^{0}\nabla\bar{u}_{1}^{2},\nabla\bar{u}_{1}^{3}\right)
&=\begin{pmatrix}
                      \lambda\partial_{x_{2}}\bar{u} & \mu\partial_{x_{1}}\bar{u} \\
                      \mu\partial_{x_{1}}\bar{u} & (\lambda+2\mu)\partial_{x_{2}}\bar{u} \\
                    \end{pmatrix}:\begin{pmatrix}
                      x_{2}\partial_{x_{1}}\bar{u} & \bar{u}+x_{2}\partial_{x_{2}}\bar{u} \\
                      -\bar{u}-x_{1}\partial_{x_{1}}\bar{u} & -x_{1}\partial_{x_{2}}\bar{u} \\
                    \end{pmatrix}\\
&=(\lambda+\mu)x_{2}\partial_{x_{1}}\bar{u}\partial_{x_{2}}\bar{u}-\mu\,x_{1}(\partial_{x_{1}}\bar{u})^{2}
-(\lambda+2\mu)x_{1}(\partial_{x_{2}}\bar{u})^{2}.
\end{align*}
Hence, using \eqref{nablau_bar}, we have
\begin{align*}
&\int_{\Omega_{R/2}}\left(\mathbb{C}^{0}\nabla\bar{u}_{1}^{2},\nabla\bar{u}_{1}^{3}\right)dx\\
&=-(\lambda+2\mu)\int_{\Omega_{R/2}}x_{1}(\partial_{x_{2}}\bar{u})^{2}dx
+O(1)\\
&=-(\lambda+2\mu)\int_{|x_1|<R/2}x_{1}\left(\frac{1}{\epsilon+h_{1}(x_{1})-h_{2}(x_{1})}-\frac{1}{\epsilon+\frac{1}{2}(h''_{1}(0)-h''_{2}(0))x_{1}^{2}}
\right)dx_{1}+O(1)\\
&=O(1).
\end{align*}
Therefore
$$\left|a_{11}^{23}\right|\leq\,C.$$
Lemma \ref{lem_a_11} is proved.
\end{proof}

\begin{proof}[Proof of Proposition \ref{prop_C1-C2}]
By \eqref{aaa1}, Lemma \ref{lem_ap} and Lemma \ref{lem_a_11},
$$
C_{1}^{1}-C_{2}^{1}=
\frac{1}{\det{a_{11}}}\left(\left(p^{1}a_{11}^{22}a_{11}^{33}-p^{3}a_{11}^{22}a_{11}^{13}\right)+O(\frac{1}{\epsilon^{1/4}})\right).
$$
Therefore
$$\left|C_{1}^{1}-C_{2}^{1}\right|\leq\,C\sqrt{\epsilon}.$$
Similarly, using \eqref{aaa2},
$$
 C_{1}^{2}-C_{2}^{2}
=\frac{1}{\det{a_{11}}}\left(\left(p^{2}a_{11}^{11}a_{11}^{33}-p^{3}a_{11}^{11}a_{11}^{23}\right)+O(\frac{1}{\epsilon^{1/4}})\right).
$$
Therefore
$$\left|C_{1}^{2}-C_{2}^{2}\right|\leq\,C\sqrt{\epsilon}.$$
The proof is completed.
\end{proof}

\begin{proof}[Proof of Proposition \ref{prop_gradient}]
Estimates \eqref{mainev3} and \eqref{mainev1+23}  have been proved
in Lemma \ref{lem_v3_v1+v2}; estimate \eqref{mainev1}
has been proved in Corollary \ref{cor1};
estimate \eqref{mainevi3} has been proved in Lemma \ref{lemvi3} and Lemma
\ref{lemma3.6}; estimate \eqref{maineC}  has been proved in Lemma
\ref{lem_C1C2bound}; and
estimate \eqref{maineC1-C2}  has been proved in
Proposition \ref{prop_C1-C2}.
The proof of Proposition \ref{prop_gradient} is completed.
\end{proof}

\section{More general $D_{1}$ and $D_{2}$}

As mentioned in the introduction,
the strict convexity assumption
 on $\partial{D}_{1}$ and $\partial{D}_{2}$ can be weakened.
In fact, our proof of Theorem \ref{mainthm1} applies,
 with minor modification, to more general situations.

In $\mathbb{R}^{2}$, under the same assumptions in the beginning of Section \ref{sec_gradient} except for the strict convexity condition, $\partial{D}_{i}$ near $P_{i}$ can be represented by the graphs of $x_{2}=\frac{\epsilon}{2}+h_{1}(x_{1})$, and $x_{2}=-\frac{\epsilon}{2}+h_{2}(x_{1})$, for $|x_{1}|<2R$. We assume that  $h_{1},h_{2}\in{C}^{2}([-2R,2R])$ and \eqref{h1h20} still holds. Instead of the
convexity assumption,
 we assume that
\begin{equation}\label{h1h24*}
\Lambda_{0}|x_{1}|^{m}\leq\,h_{1}(x_1)-h_{2}(x_1)\leq\Lambda_{1}|x_{1}|^{m},\quad\mbox{for}~~|x_{1}|<2R,
\end{equation}
and
\begin{equation}\label{h1h25*}
|h'_{i}(x_1)|\leq\,C|x_{1}|^{m-1},\quad|h''_{i}(x_1)|\leq\,C|x_{1}|^{m-2},\quad\,i=1,2,\quad\mbox{for}~~|x_{1}|<2R,
\end{equation}
for some $\epsilon-$independent constants $0<\Lambda_{0}<\Lambda_{1}$, and $m\geq2$. Define $\delta:=\delta(z_{1})$ as \eqref{delta}.
Clearly,
\begin{equation}\label{delta_z1*}
\frac{1}{C}(\epsilon+|z_{1}|^{m})\leq\delta(z_{1})\leq\,C(\epsilon+|z_{1}|^{m}).
\end{equation}
Then
\begin{theorem}\label{mainthm1*}
Under the above assumptions with $m\geq2$, let $u\in{H}^{1}(\Omega;\mathbb{R}^{2})\cap{C}^{1}(\overline{\widetilde{\Omega}};\mathbb{R}^{2})$ be
a solution to \eqref{mainequation}. Then for $0<\epsilon<1$,
 we have
\begin{equation}\label{mainestimates*}
|\nabla{u}(x)|\leq
\left\{
\begin{array}{ll}
\displaystyle{
C\frac{\epsilon^{1-\frac{1}{m}}+\mathrm{dist}(x,\overline{P_{1}P_{2}})}{\epsilon+\mathrm{dist}^{m}(x,\overline{P_{1}P_{2}})}\|\varphi\|_{C^{1,\gamma}
(\partial{\Omega};\mathbb{R}^{2})},
}
&\quad\,x\in
\widetilde \Omega,\\
&\\
C\|\varphi\|_{C^{1,\gamma}
(\partial{\Omega};\mathbb{R}^{2})},&\quad x\in D_1\cup D_2.
\end{array}
\right.
\end{equation}
where $C$ is a universal constant.
In particular,
\begin{equation}\label{normbound*}
\|\nabla u\|_{ L^\infty(\Omega) }\le C\epsilon^{ \frac{1}{m}-1}
\|\varphi\|_{C^{1,\gamma}(\partial{\Omega};\mathbb{R}^{2})}.
\end{equation}
\end{theorem}

In the following, we only list the main differences. We define $\bar{u}$ by \eqref{ubar} as before.
A calculation gives
\begin{equation}\label{nablau_bar*}
|\partial_{x_{1}}\bar{u}(x)|\leq\frac{C|x_{1}|^{m-1}}{\epsilon+|x_{1}|^{m}},\quad
|\partial_{x_{2}}\bar{u}(x)|\leq\frac{C}{\epsilon+|x_{1}|^{m}},\quad~x\in\Omega_{R},
\end{equation}
by \eqref{h1h2}, we have
\begin{equation}\label{nabla2u_bar*}
|\partial_{x_{1}x_{1}}\bar{u}(x)|\leq\frac{C|x_{1}|^{m-2}}{\epsilon+|x_{1}|^{m}},~~
|\partial_{x_{1}x_{2}}\bar{u}(x)|\leq\frac{C|x_{1}|^{m-1}}{(\epsilon+|x_{1}|^{m})^{2}},
~~\partial_{x_{2}x_{2}}\bar{u}(x)=0,~~x\in\Omega_{R}.
\end{equation}
Define $\bar{u}_{i}^{\alpha}$, $i,\alpha=1,2$ as in \eqref{def:ubar1112} and \eqref{def:ubar2122}. By \eqref{L_u}, \eqref{nablau_bar*} and \eqref{nabla2u_bar*}, we have
\begin{equation}\label{L_ubar_ialpha*}
|\mathcal{L}_{\lambda,\mu}\bar{u}_{i}^{\alpha}(x)|\leq\frac{C|x_{1}|^{m-2}}{\epsilon+|x_{1}|^{m}}
+\frac{C|x_{1}|^{m-1}}{(\epsilon+|x_{1}|^{m})^{2}},\quad\,i,\alpha=1,2,\quad~x\in\Omega_{R}.
\end{equation}

Instead of Proposition \ref{prop_gradient}, we have

\begin{prop}\label{prop_gradient*}
Under the hypotheses of Theorem \ref{mainthm1*}
and a normalization
 $\|\varphi\|_{C^{1,\gamma}(\partial\Omega)}=1$,
we have, for $0<\epsilon<1$,
\begin{equation}\label{mainev3*}
\|\nabla{v}_{3}\|_{L^{\infty}(\widetilde{\Omega})}\leq\,C;
\end{equation}
\begin{equation}\label{mainev1+23*}
\|\nabla{v}_{1}^{\alpha}+\nabla{v}_{2}^{\alpha}\|_{L^{\infty}(\widetilde{\Omega})}\leq\,C,\quad\alpha=1,2, 3;
\end{equation}
\begin{equation}\label{mainev1*}
|\nabla{v}_{i}^{\alpha}(x)|\leq\frac{C}{\epsilon+\mathrm{dist}^{m}(x,\overline{P_{1}P_{2}})},\quad\,i,\alpha=1,2,\quad\,x\in\widetilde{\Omega};
\end{equation}
\begin{equation}\label{mainevi3*}
|\nabla{v}_{i}^{3}(x)|\leq\,C\frac{\epsilon+\mathrm{dist}(x,\overline{P_{1}P_{2}})}{\epsilon+\mathrm{dist}^{m}(x,\overline{P_{1}P_{2}})},\quad\,i=1,2,
\quad\,x\in\widetilde{\Omega};
\end{equation}
and
\begin{equation}\label{maineC*}
|C_{i}^{\alpha}|\leq\,C,\quad
i=1,2,~\alpha=1,2,3;
\end{equation}
\begin{equation}\label{maineC1-C2*}
\quad|C_{1}^{\alpha}-C_{2}^{\alpha}|\leq\,C\epsilon^{1-\frac{1}{m}},\quad
\alpha=1,2.
\end{equation}
\end{prop}

Denote
$$w_{i}^{\alpha}:=v_{i}^{\alpha}-\bar{u}_{i}^{\alpha},\quad\,i=1,2,~\alpha=1,2,3.$$
Then, instead of Proposition \ref{prop1}, we have

\begin{prop}\label{prop1*}
Assume the above, let $v_{i}^{\alpha}\in{C}^2(\widetilde{\Omega};\mathbb{R}^{2})\cap{C}^1(\overline{\widetilde{\Omega}};\mathbb{R}^{2})$ be the
weak solution of \eqref{v1alpha}. Then, for $i,\alpha=1,2$,
\begin{equation}\label{energy_w*}
\int_{\widetilde{\Omega}}\left|\nabla{w}_{i}^{\alpha}\right|^{2}dx\leq\,C,
\end{equation}
\begin{equation}\label{energy_w_inomega_z1*}
\int_{\widehat{\Omega}_{\delta}(z_{1})}\left|\nabla{w}_{i}^{\alpha}\right|^{2}dx\leq
\begin{cases}C\left(\epsilon^{2m-2}+|z_{1}|^{2m-2}\right),&|z_{1}|\leq\sqrt[m]{\epsilon},\\
C|z_{1}|^{2m-2},&\sqrt[m]{\epsilon}<|z_{1}|\leq\,R,
\end{cases}
\end{equation}
and
\begin{equation}\label{nabla_w_ialpha*}
\left|\nabla{w}_{i}^{\alpha}(x)\right|\leq
\begin{cases}C
\frac{\epsilon^{m-1}+|x_1|^{m-1}}{\epsilon},&|x_{1}|\leq\sqrt[m]{\epsilon},\\
\frac{C}{|x_{1}|},&\sqrt[m]{\epsilon}<|x_{1}|\leq\,R.
\end{cases}
\end{equation}
\end{prop}

\begin{proof}
The proof of \eqref{energy_w*} is the same as that of \eqref{energy_w}. We only list the main differences from STEP 2 and STEP 3 in the proof of Proposition \ref{prop1}.

\noindent{\bf STEP 2.} Proof of \eqref{energy_w_inomega_z1*}.

\noindent
{\bf Case 1. For $\sqrt[m]{\epsilon}\leq|z_{1}|\leq\,R/2$.}

Note that for $0<s<\frac{2|z_{1}|}{3}$, we have
\begin{align}\label{energy_w_square*}
\int_{\widehat{\Omega}_{s}(z_{1})}|w|^{2}dx
&\leq\int_{|x_{1}-z_{1}|\leq\,s}(\epsilon+h_{1}(x_{1})
-h_{2}(x_{1}))^{2}\int_{-\frac{\epsilon}{2}+h_{2}(x_{1})}^{\frac{\epsilon}{2}+h_{1}(x_{1})}
|\partial_{x_{2}}w(x_{1},x_{2})|^{2}dx_{2}dx_{1}\nonumber\\
&\leq\,C|z_{1}|^{2m}\int_{\widehat{\Omega}_{s}(z_{1})}|\nabla{w}|^{2}dx,
\end{align}
By \eqref{L_ubar_ialpha*}, we have
\begin{align}\label{integal_Lubar11*}
\int_{\widehat{\Omega}_{s}(z_{1})}\left|\mathcal{L}_{\lambda,\mu}\bar{u}_{1}^{1}\right|^{2}
dx
&\leq\int_{\widehat{\Omega}_{s}(z_{1})}\left(\frac{C|x_{1}|^{m-2}}{\epsilon+|x_{1}|^{m}}
+\frac{C|x_{1}|^{m-1}}{(\epsilon+|x_{1}|^{m})^{2}}\right)^{2}dx\nonumber\\
&\leq\frac{C|z_{1}|^{m}s}{|z_{1}|^{2(m+1)}}\leq\frac{Cs}{|z_{1}|^{m+2}},\quad\quad\,0<s<\frac{2|z_{1}|}{3}.
\end{align}
As before, it follows from the above and \eqref{FsFt11} that
\begin{equation}\label{tildeF111*}
\widehat{F}(t)\leq\,\left(\frac{C_{0}|z_{1}|^{m}}{s-t}\right)^{2}\widehat{F}(s)+C(s-t)^{2}\frac{s}{|z_{1}|^{m+2}},
\quad\forall~0<t<s<\frac{2|z_{1}|}{3},
\end{equation}
where $C_0$ is also a universal constant.

Let $t_{i}=2C_{0}i\,|z_{1}|^{m}$, $i=1,2,\cdots$. Then
$$\frac{C_{0}|z_{1}|^{m}}{t_{i+1}-t_{i}}=\frac{1}{2}.$$
Let $k=\left[\frac{1}{4C_{0}|z_{1}|^{m-1}}\right]$. Then by \eqref{tildeF111*} with $s=t_{i+1}$ and $t=t_{i}$, we have
$$\widehat{F}(t_{i})\leq\,\frac{1}{4}\widehat{F}(t_{i+1})+\frac{C(t_{i+1}-t_{i})^{2}t_{i+1}}{|z_{1}|^{m+2}}
\leq\,\frac{1}{4}\widehat{F}(t_{i+1})+C(i+1)|z_{1}|^{2m-2},
$$
After $k$ iterations, we have,
using (\ref{energy_w*}),
\begin{eqnarray*}
\widehat{F}(t_{1}) &\leq &\left(\frac{1}{4}\right)^{k}\widehat{F}(t_{k+1})+C|z_{1}|^{2m-2}\sum_{l=1}^{k}\left(\frac{1}{4}\right)^{l-1}(l+1)
\\
&\leq &C|z_{1}|^{2m-2}.
\end{eqnarray*}
This implies that
$$\int_{\widehat{\Omega}_{\delta}(z_{1})}|\nabla{w}|^{2}dx\leq\,C|z_{1}|^{2m-2}.$$

\noindent
{\bf Case 2.} For $|z_{1}|\leq \sqrt[m]{\epsilon}$.

 For $0<t<s<\sqrt[m]{\epsilon}$, estimate \eqref{energy_w_square*} becomes
 \begin{align}\label{energy_w_square_in*}
\int_{\widehat{\Omega}_{s}(z_{1})}|w|^{2}dx
\leq\,C\epsilon^{2}\int_{\widehat{\Omega}_{s}(z_{1})}|\nabla{w}|^{2}dx,
\quad\,0<s<\sqrt[m]{\epsilon};
\end{align}
Estimate \eqref{integal_Lubar11*} becomes
\begin{align}\label{integal_Lubar11_in*}
\int_{\widehat{\Omega}_{s}(z_{1})}
|\mathcal{L}_{\lambda,\mu}\bar{u}_{1}^{1}|^{2} dx
&\leq\int_{\widehat{\Omega}_{s}(z_{1})}
\left(\frac{C|x_{1}|^{m-2}}{\epsilon+|x_{1}|^{m}}
+\frac{C|x_{1}|^{m-1}}{(\epsilon+|x_{1}|^{m})^{2}}\right)^{2} dx\nonumber\\
&\leq\frac{Cs}{\epsilon}
+\frac{C(|z_{1}|^{2m-2}+s^{2m-2})s}{\epsilon^{3}},\quad\mbox{for}~\,0<s<\sqrt[m]{\epsilon};
\end{align}
Estimate \eqref{tildeF111*} becomes, in view of \eqref{FsFt11},
\begin{equation}\label{tildeF111_in*}
\widehat{F}(t)\leq\,\left(\frac{C_{0}\epsilon}{s-t}\right)^{2}\widehat{F}(s)+C(s-t)^{2}s(\frac{1}{\epsilon}+\frac{|z_{1}|^{2m-2}}{\epsilon^{3}}
+\frac {s^{2m-2}}{ \epsilon^3}),
\quad\forall~0<t<s<\sqrt[m]{\epsilon}.
\end{equation}
Let $t_{i}=2C_{0}i\epsilon$, $i=1,2,\cdots$. Then
$$\frac{C_{0}\epsilon}{t_{i+1}-t_{i}}=\frac{1}{2}.$$
Let $k=\left[\frac{1}{4C_{0}\epsilon^{1-\frac{1}{m}}}\right]$. Then by \eqref{tildeF111_in} with $s=t_{i+1}$ and $t=t_{i}$, we have
$$\widehat{F}(t_{i})\leq\,\frac{1}{4}\widehat{F}(t_{i+1})
+Ci^3\left(\epsilon^{2m-2}+|z_{1}|^{2m-2}\right).$$
After $k$ iterations, we have,
using (\ref{energy_w*}),
\begin{eqnarray*}
\widehat{F}(t_{1})
&\leq& \left(\frac{1}{4}\right)^{k}\widehat{F}(t_{k+1})
+C\sum_{l=1}^{k}\left(\frac{1}{4}\right)^{l-1}l^3\left(\epsilon^{2m-2}+|z_{1}|^{2m-2}\right)\\
&\leq&\,
C\left(\frac{1}{4}\right)^{\frac{1}{C\epsilon^{1-\frac{1}{m}}}}
+C\left(\epsilon^{2m-2}+|z_{1}|^{2m-2}\right)
\leq\,C\left(\epsilon^{2m-2}+|z_{1}|^{2m-2}\right).
\end{eqnarray*}
This implies that
$$\int_{\widehat{\Omega}_{\delta}(z_{1})}|\nabla{w}|^{2}dx\leq\,C\left(\epsilon^{2m-2}+|z_{1}|^{2m-2}\right).$$

{\bf STEP 3.} Proof of \eqref{nabla_w_ialpha*}.

Using a change of variables \eqref{changeofvariant}, define $Q'_{r}$, $\hat{h}_{1}$, and $\hat{h}_{2}$ as in the proof of Proposition \ref{prop1}. Then
by \eqref{h1h25*},
$$|\hat{h}_{1}'(0)|+|\hat{h}_{2}'(0)|\leq\,C|z_{1}|^{m-1},
\quad|\hat{h}_{1}''(0)|+|\hat{h}_{2}''(0)|\leq\,C\delta|z_{1}|^{m-2}.$$
Since $R$ is small, $\|\hat{h}_{1}\|_{C^{1,1}((-1,1))}$ and $\|\hat{h}_{2}\|_{C^{1,1}((-1,1))}$ are small and $\frac{1}{2}Q'_{1}$ is essentially a unit square as far as applications of Sobolev embedding
theorems and classical $L^{p}$ estimates for elliptic systems are concerned.
By the same argument as in the proof of Proposition \ref{prop1}, \eqref{AAA} still holds. We divide into two cases to proceed.

\noindent
{\bf Case 1.}\   For $\sqrt[m]{\epsilon}\leq|z_{1}|\leq\,R/2$.

By (\ref{energy_w_inomega_z1*}),
$$\int_{\widehat{\Omega}_{\delta}(z_{1})}\left|\nabla{w}_{1}^{1}\right|^{2}dx\leq\,C|z_{1}|^{2m-2}.$$
By \eqref{L_ubar_ialpha*},
$$\delta^{2}\left|\mathcal{L}_{\lambda,\mu}\bar{u}_{1}^{1}\right|\leq\delta^{2}\left(\frac{C}{|z_{1}|^{2}}+\frac{C}{|z_{1}|^{m+1}}\right)\leq\,C|z_{1}|^{m-1},
\qquad \mbox{in}\ \widehat\Omega_\delta(z_1).$$
We deduce from (\ref{AAA}) that
$$\left|\nabla{w}_{1}^{1}(z_{1},x_{2})\right|=\frac{C|z_{1}|^{m-1}}{\delta}\leq\frac{C}{|z_{1}|},
\qquad\forall \
-\frac \epsilon 2 +h_2(z_1)<
x_2< \frac \epsilon 2 +h_1(z_1).$$

\noindent
{\bf Case 2.}\  For $|z_{1}|\leq2\sqrt[m]{\epsilon}$.

By (\ref{energy_w_inomega_z1*}),
$$\int_{\widehat{\Omega}_{\delta}(z_{1})}\left|\nabla{w}_{1}^{1}\right|^{2}dx\leq\,C(\epsilon^{2m-2}+|z_{1}|^{2m-2}).$$
By \eqref{L_ubar_ialpha*},
\begin{equation}
\delta^{2}\left|\mathcal{L}_{\lambda,\mu}\bar{u}_{1}^{1}\right|\leq
C \delta^{2}\left(\frac{(\epsilon+|z_1|)^{m-2}}{\epsilon}
+\frac{(\epsilon+|z_1|)^{m-1}}{\epsilon^2}\right)\leq\,C
\left(\epsilon+|z_1|\right)^{m-1}, \quad
\mbox{in}~~
\widehat \Omega_\delta(z_1).
\end{equation}
We deduce from (\ref{AAA}) that
$$\left|\nabla{w}_{1}^{1}(z_{1},x_{2})\right|=\frac{C}{\delta}
\left(\epsilon^{m-1}+|z_1|^{m-1}\right)\leq
C \frac{\epsilon^{m-1} +|z_1|^{m-1}}{\epsilon},
\qquad\forall \
-\frac \epsilon 2 +h_2(z_1)<
x_2< \frac \epsilon 2 +h_1(z_1).$$
Proposition \ref{prop1*} is established.
\end{proof}

Define $\bar{u}_{i}^{3}$, $i=1,2$ by \eqref{def:ubar1323}.
Using \eqref{h1h24*}, \eqref{h1h25*} and  \eqref{nablau_bar*}, we  have
\begin{equation}\label{nabla_baru13_in*}
\left|\nabla\bar{u}_{i}^{3}(x)\right|\leq\,
\frac{C(\epsilon+|x_1|)}
{\epsilon+|x_{1}|^{m}},\quad\,i=1,2,\quad\,x\in\Omega_{R},
\end{equation}
and
\begin{equation}\label{nabla_baru13_out*}
\left|\nabla\bar{u}_{i}^{3}(x)\right|\leq\,C,\quad\,i=1,2,\quad\,x\in\widetilde{\Omega}\setminus\Omega_{R}.
\end{equation}
It follows from \eqref{L_u}, \eqref{nablau_bar*} and \eqref{nabla2u_bar*} that
\begin{equation}\label{L_ubar_i3*}
\left|\mathcal{L}_{\lambda,\mu}\bar{u}_{i}^{3}\right|\leq\frac{C}{\epsilon+|x_{1}|^{m}},\quad\,i=1,2,\quad\,x\in\Omega_{R}.
\end{equation}
Then Lemma \ref{lemvi3} still holds, while Lemma \ref{lem_I11} and Lemma \ref{lemma3.6} become
\begin{lemma}\label{lem_I11*}
With $\delta=\delta(z_{1})$ in \eqref{delta}, we have,
 for $i=1,2$,
\begin{equation}\label{energyw13_out*}
\int_{\widehat{\Omega}_{\delta}(z_{1})}\left|\nabla{w}_{i}^{3}\right|^{2}dx\leq
\begin{cases}
C\epsilon^{2},& |z_{1}|<\sqrt[m]{\epsilon},\\
C|z_{1}|^{2m},&\sqrt[m]{\epsilon}\leq|z_{1}|<R/2.
\end{cases}
\end{equation}
\end{lemma}

\begin{proof}
The proof is very similar to that of Lemma \ref{lem_I11}. By the same argument, we still have \eqref{FsFt13} holds.

\noindent
{\bf Case 1. $\sqrt[m]{\epsilon}<|z_{1}|<R/2$.}

We still have \eqref{energy_w_square*} for $0<s<\frac{2|z_{1}|}{3}$. Instead of \eqref{integal_Lubar11*}, we have, using
 \eqref{L_ubar_i3*},
\begin{align}\label{Lu13*}
\int_{\widehat{\Omega}_{s}(z_{1})}\left|\mathcal{L}_{\lambda,\mu}\bar{u}_{1}^{3}\right|^{2}dx\leq\frac{Cs}{|z_{1}|^{m}}.
\end{align}
Instead of \eqref{tildeF111*}, we have
\begin{equation}\label{tildeF13*}
\widehat{F}(t)\leq\,\left(\frac{C_{0}|z_{1}|^{m}}{s-t}\right)^{2}\widehat{F}(s)+C(s-t)^{2}\frac{s}{|z_{1}|^{m}},
\quad\forall~0<t<s<\frac{2|z_{1}|}{3}.
\end{equation}
We define $\{t_{i}\}$, $k$ and iterate as in the proof of \eqref{energy_w_inomega_z1*}, right below formula \eqref{tildeF111*}, to obtain,
using \eqref{XYZ},
$$\widehat{F}(t_{1})
\leq\,\left(\frac{1}{4}\right)^{k}\widehat{F}(\frac{2|z_{1}|}{3})+C|z_{1}|^{2m}\sum_{l=1}^{k}\left(\frac{1}{4}\right)^{l}l\leq\,C|z_{1}|^{2m}.$$
This implies that
$$\int_{\widehat{\Omega}_{\delta}(z_{1})}|\nabla{w}|^{2}dx\leq\,C|z_{1}|^{2m}.$$

\noindent
{\bf Case 2. $|z_{1}|<\sqrt[m]{\epsilon}$.}

Estimate \eqref{energy_w_square_in*} remains the same. Estimate \eqref{integal_Lubar11_in*} becomes
\begin{align}\label{integal_Lubar13_in*}
\int_{\widehat{\Omega}_{s}(z_{1})}\left|\mathcal{L}_{\lambda,\mu}\bar{u}_{1}^{3}\right|^{2}dx
\leq\frac{Cs}{\epsilon},\quad\,0<s<\sqrt[m]{\epsilon}.
\end{align}
Estimate \eqref{tildeF111_in*} becomes
\begin{equation}\label{tildeF2*}
\widehat{F}(t)\leq\,\left(\frac{C_{0}\epsilon}{s-t}\right)^{2}\widehat{F}(s)+\frac{C(s-t)^{2}s}{\epsilon},
\quad\forall~0<t<s<\sqrt[m]{\epsilon}.
\end{equation}
Define $\{t_{i}\}$, $k$ and iterate as in the proof of \eqref{energy_w_inomega_z1}, right below formula \eqref{tildeF111_in}, to obtain
$$\widehat{F}(t_{1})\leq\,\left(\frac{1}{4}\right)^{k}\widehat{F}(t_{k+1})
+C\sum_{l=1}^{k}\left(\frac{1}{4}\right)^{l-1}l\epsilon^{2}
\leq\,C\epsilon^{2}.$$
This implies that
$$\int_{\widehat{\Omega}_{\delta}(z_{1})}|\nabla{w}|^{2}dx\leq\,C\epsilon^{2}.$$
\end{proof}
It is not difficult to obtain
\begin{lemma}\label{lemma3.6*}
\begin{equation}\label{nabla_w13*}
\|\nabla{w}_{i}^{3}\|_{L^{\infty}(\widetilde{\Omega})}\leq\,C,\quad\,i=1,2.
\end{equation}
Consequently,
\begin{equation}\label{nabla_v13*}
|\nabla{v}_{i}^{3}(x)|\leq\,\frac{C(\epsilon+|x_{1}|)}{\epsilon+|x_{1}|^{m}},\quad\,i=1,2,\quad\,x\in{\Omega_R}.
\end{equation}
\end{lemma}

The last main difference is the computation of $a_{11}^{\alpha\alpha}$, $\alpha=1,2$. In fact,
By   \eqref{coeff4_strongelyconvex},
\eqref{coeff2}, \eqref{mainev1} and \eqref{energy_w*},
\begin{align*}
a_{11}^{\alpha\alpha}&
=\int_{\widetilde{\Omega}}\left(\mathbb{C}^{0}e\left(v_{1}^{\alpha}\right),e\left(v_{1}^{\alpha}\right)\right)dx
=\int_{\widetilde{\Omega}}\left(\mathbb{C}^{0}\nabla{v}_{1}^{\alpha},\nabla{v}_{1}^{\alpha}\right)dx\\
&\leq\,C\int_{\widetilde{\Omega}}\left|\nabla{v}_{1}^{\alpha}\right|^{2}dx
\leq\,C\int_{\widetilde{\Omega}}\left|\nabla\bar{u}_{1}^{\alpha}\right|^{2}dx+C\int_{\widetilde{\Omega}}\left|\nabla{w}_{1}^{\alpha}\right|^{2}dx\\
&\leq\,C\int_{-R}^{R}\frac{1}{\epsilon+h_{1}(x_{1})-h_{2}(x_{1})}dx_{1}+C\\
&\leq\,C\int_{0}^{R}\frac{1}{\epsilon+|x_{1}|^{m}}dx_{1}+C\\
&\leq\,C\epsilon^{\frac{1}{m}-1},\qquad\,\alpha=1,2.
\end{align*}
Using \eqref{energy_w*} again, we have
\begin{eqnarray*}
a_{11}^{11}&=&
\int_{\widetilde{\Omega}}\left(\mathbb{C}^{0}e\left(v_{1}^{1}\right),
e\left(v_{1}^{1}\right)\right)dx
\geq\frac{1}{C}\int_{\widetilde{\Omega}}\left|e\left(v_{1}^{1}\right)
\right|^{2}dx\\
&\geq&
\frac{1}{2C}\int_{\widetilde{\Omega}}\left|e\left(\bar u_{1}^{1}\right)
\right|^{2}dx
-C \int_{\widetilde{\Omega}}\left|e\left(w_{1}^{1}\right)
\right|^{2}dx\\
&\geq&
\frac 1{2C}
\int_{\widetilde{\Omega}}\left|e\left(\bar u_{1}^{1}\right)
\right|^{2}dx
-C.
\end{eqnarray*}
In view of \eqref{inequality}, we have
\begin{align*}
\int_{\widetilde{\Omega}}
\left|e\left(\bar u_{1}^{1}\right)
\right|^{2}
dx
&\geq
\frac 14
\int_{\widetilde\Omega} |\partial_{x_2}\bar u|^2 dx
\ge \frac{1}{C}\int_{\Omega_R} \frac{dx}
{  (\epsilon+h_1(x_1)-h_2(x_1))^2}\\
&\geq\frac{1}{C}\int_{0}^{R}\frac{1}{\epsilon+|x_{1}|^{m}}dx_{1}+C\\
&\geq\frac{\epsilon^{\frac{1}{m}-1}}{C}.
\end{align*}
Thus
$$
a_{11}^{11}
\geq\frac{\epsilon^{\frac{1}{m}-1}}{C}.
$$
Similarly, we have
$$
a_{11}^{22}
\geq\frac{\epsilon^{\frac{1}{m}-1}}{C}.
$$
By the argument as in the proof of Lemma \ref{lem_a_11}, we have
$$\frac{\epsilon^{\frac{2}{m}-2}}{C}\leq\det{a}_{11}\leq\,C\epsilon^{\frac{2}{m}-2}.$$
Then, we have
$$|C_{1}^{\alpha}-C_{2}^{\alpha}|\leq\,C\epsilon^{1-\frac{1}{m}},\quad\alpha=1,2.$$
The proof of Theorem \ref{mainthm1*} is finished.

\section{Appendix: Some results on the Lam\'{e} system with infinity coefficients}\label{sec_appendix}

Assume that in $\mathbb{R}^{d}$, $\Omega$ and $\omega$ are bounded open sets with smooth boundaries satisfying
$$\overline{\omega}=\cup_{s=1}^{m}\overline{\omega}_{s}\subset\Omega,$$
where $\{\omega_{s}\}$ are connected components of $\omega$. Clearly, $m<\infty$ and $\omega_{s}$ is open for all $1\leq{s}\leq{m}$. Given
$\varphi\in{C}^{1, \gamma}(\partial\Omega;\mathbb{R}^{d})$,
$0<\gamma<1$,
 $\mu>0$, $d\lambda+2\mu>0$, and $$\mu_{n}^{(s)}\rightarrow\infty,~~d\lambda_{n}^{(s)}+2\mu_{n}^{(s)}\rightarrow\infty,\quad\mbox{as}~~n\rightarrow\infty.$$ We denote
$$\mathbb{C}_{n}^{(s)}
:=\lambda_{n}^{(s)}\delta_{ij}\delta_{kl}+\mu_{n}^{(s)}\left(\delta_{ik}\delta_{jl}+\delta_{il}\delta_{jk}\right),\quad\,1\leq\,s\leq\,m,$$
$$\mathbb{C}^{(0)}:
=\lambda\delta_{ij}\delta_{kl}+\mu\left(\delta_{ik}\delta_{jl}+\delta_{il}\delta_{jk}\right),\quad\quad\quad\quad\quad\quad\quad$$
and
$$\mathbb{C}_{n}(x)=
\begin{cases}
\mathbb{C}_{n}^{(s)},&\mbox{in}~~\omega_{s},~~1\leq\,s\leq\,m,\\
\mathbb{C}^{(0)},&\mbox{in}~~\Omega\setminus\overline{\omega}.
\end{cases}$$
Consider for every $n$
\begin{equation}\label{5.1}
\begin{cases}
\nabla\cdot\left(\mathbb{C}_{n}e(u_{n})\right)=0,
&\mbox{in}~\Omega,\\
u=\mathbf{\varphi},&\mbox{on}~\partial{\Omega}.
\end{cases}
\end{equation}

Let $\Psi$ be the linear space of
rigid displacements of $\mathbb{R}^{d}$, i.e.
the set of all vector -valued functions
$\eta=(\eta^1, \cdots, \eta^d)^T$ such that
$\eta=a+Ax$, where $a=(a_1, \cdots, a_d)^T$
is a vector with constant real components, $A$ is
a skew-symmetric $(d\times d)$-matrix
with real constant elements. It is easy to see that
$\Psi$ is a linear space of dimension
$d(d+1)/2$.
Denote
$$
\Psi=\mathrm{span}
\left\{~\psi^\alpha\ |\ \ 1\le \alpha \le \frac {d(d+1) }2 \right\}.
$$

Equation \eqref{5.1} can be rewritten in the following form to emphasize the transmission condition on $\partial\omega$:
\begin{equation}\label{5.2}
\begin{cases}
\nabla\cdot\left(\mathbb{C}_{n}^{(s)}e(u_{n})\right)=0,&\mbox{in}~\omega_{s},~1\leq\,s\leq\,m,\\\\
\nabla\cdot\left(\mathbb{C}^{(0)}e(u_{n})\right)=0,&\mbox{in}~
\Omega\setminus\overline \omega,\\\\
\dfrac{\partial{u}_{n}}{\partial\nu_{0}}\Big|_{+}\cdot\psi^{\alpha}=
\dfrac{\partial{u}_{n}}{\partial\nu_{0}}\Big|_{-}\cdot\psi^{\alpha},&\mbox{on}~\partial\omega_{s},~1\leq\,s\leq\,m;\ \
1\le \alpha\le \frac{ d(d+1)}2,\\
\end{cases}
\end{equation}
where
$$\frac{\partial{u}_{n}}
{\partial\nu_{0}}\bigg|_{+}:=\left(\mathbb{C}^{(0)}e(u)\right)\vec{n}
=\lambda\left(\nabla\cdot{u}_{n}\right)\vec{n}+\mu\left(\nabla{u}_{n}
+(\nabla{u}_{n})^{T}\right)\vec{n},\quad\mbox{on}~\partial\omega_{s},$$
$$\frac{\partial{u}_{n}}{\partial\nu_{0}}\bigg|_{-}:=\left(\mathbb{C}_{n}^{(s)}e(u)\right)\vec{n}
=\lambda_{n}^{(s)}\left(\nabla\cdot{u}_{n}\right)\vec{n}+
\mu_{n}^{(s)}\left(\nabla{u}_{n}+(\nabla{u}_{n})^{T}\right)\vec{n},
\quad\mbox{on}~\partial\omega_{s},$$
and the subscript $\pm$ indicates the limit from outside
 and inside $\omega_{s}$, respectively.

\begin{theorem}
If $u_{n}\in{H}^{1}(\Omega
;\mathbb{R}^{d}
)$ is a solution of equation \eqref{5.1}, then
$u_{n}\in{C}^{1}(\overline{\Omega\setminus\omega}
;\mathbb{R}^{d}
)\cap{C}^{1}(\overline{\omega}
;\mathbb{R}^{d}
)$ and satisfies equation \eqref{5.2}.

If $u_{n}\in{C}^{1}(\overline{\Omega\setminus\omega}
;\mathbb{R}^{d}
)\cap{C}^{1}(\overline{\omega}
;\mathbb{R}^{d}
)$ is a solution of equation \eqref{5.2}, then $u_{n}\in{H}^{1}(\Omega
;\mathbb{R}^{d}
)$ and satisfies equation \eqref{5.1}.
\end{theorem}

\begin{proof}
The first part of the theorem follows from Proposition 1.4 of \cite{ln}. The proof of the rest is standard.
\end{proof}

\begin{theorem}
There exists at most one solution $u_{n}\in{H}^{1}(\Omega
;\mathbb{R}^{d}
)$
 to equation \eqref{5.1}.
\end{theorem}

\begin{proof}
We only need to prove that if $\varphi=0$ then a solution $u_{n}$ of \eqref{5.1} is zero. Indeed it follows from \eqref{5.1} that
$$\int_{\Omega}\left(\mathbb{C}_{n}e(u_{n}),e(\psi)\right)
dx=0,\quad\forall~~\psi\in{C}_{c}^{\infty}(\Omega
;\mathbb{R}^{d}
).$$
This implies by density of $C_{c}^{\infty}(\Omega
;\mathbb{R}^{d}
)$ in $H_{0}^{1}(\Omega
;\mathbb{R}^{d}
)$ that $\int_{\Omega}\left(\mathbb{C}_{n}e(u_{n}),e(u_{n})\right)
dx=0$. By the property of $\mathbb{C}_{n}$ and the First Korn inequality, we have $\nabla{u}_{n}=0$, and therefore $u_{n}=0$.
\end{proof}

Define the functional
\begin{equation}\label{5.3}
I_{n}[v]:=\frac{1}{2}\int_{\Omega}\left(\mathbb{C}_{n}(x)e(v),e(v)\right)dx,
\end{equation}
where $v$ belongs to the set
$$H_{\varphi}^{1}(\Omega;\mathbb{R}^{d})
:=\left\{v\in{H}^{1}(\Omega;\mathbb{R}^{d})~\bigg|~v=
\varphi,~\mbox{on}~\partial\Omega~\right\},
$$
where $\varphi\in C^{1,\gamma}(\partial \Omega;\mathbb{R}^{d})$,
$0<\gamma<1$.

\begin{theorem}\label{thm_5.3}
For every $n$, there exists a minimizer $u_{n}\in{H}_{\varphi}^{1}(\Omega;\mathbb{R}^{d})$ satisfying
$$I_{n}[u_{n}]:=\min_{v\in{H}_{\varphi}^{1}(\Omega;\mathbb{R}^{d})}I_{n}[v].$$
Moreover, $u_{n}\in{H}^{1}(\Omega;\mathbb{R}^{d})$ is a solution of equation \eqref{5.1}.
\end{theorem}

The proof of Theorem \ref{thm_5.3} is standard. The existence of a minimizer $u_{n}$ follows from the lower semi-continuity property of the functional with respect to the weak convergence in $H^{1}(\Omega;\mathbb{R}^{d})$ and the First Korn inequality.

Comparing equation \eqref{5.1}, the Lam\'{e} system with infinity coefficients
is
\begin{equation}\label{5.4}
\begin{cases}
\nabla\cdot\left(\mathbb{C}^{(0)}e(u)\right)=0,&\mbox{in}~\Omega\setminus
\overline \omega,\\
u\big|_{+}=u\big|_{-},&\mbox{on}~\partial\omega,\\
e(u)=0,&\mbox{in}~\omega,\\
\int_{\partial\omega_{s}}\dfrac{\partial{u}}{\partial\nu_{0}}\Big|_{+}\cdot\psi^{\alpha}=0,&1\le s\le m; \ ~ 1\le \alpha\le \frac {d(d+1)}2,\\
u=\varphi,&\mbox{on}~\partial\Omega.
\end{cases}
\end{equation}
We have similar results:

\begin{theorem}
If $u\in{H}^{1}(\Omega;\mathbb{R}^{d})$ satisfies
 \eqref{5.4} except for the fourth line, then
$u\in{C}^{1}(\overline{\Omega\setminus\omega}
;\mathbb{R}^{d}
)\cap{C}^{1}(\overline{\omega}
;\mathbb{R}^{d}
)$.
\end{theorem}

\begin{proof}
By the third line of equation \eqref{5.4}, $u$ is a linear
combination of $\{\psi^\alpha\}$,
 and therefore $u\in{C}^{\infty}(\partial\omega)$. Since $\nabla\cdot\left(\mathbb{C}^{(0)}e(u)\right)=0$ on $\Omega\setminus\overline{\omega}$, the regularity of $u$ in $\overline{\Omega\setminus\omega}$ follows from \cite{adn}.
\end{proof}

\begin{theorem}\label{theorem5.6}
There exists at most one solution $u\in{H}^{1}(\Omega
;\mathbb{R}^{d}
)\cap{C}^{1}(\overline{\Omega\setminus\omega}
;\mathbb{R}^{d}
)\cap{C}^{1}(\overline{\omega}
;\mathbb{R}^{d}
)$ of  \eqref{5.4}.
\end{theorem}

\begin{proof}
It is equivalent to showing that if $\varphi=0$, equation \eqref{5.4} only has the solution $u=0$. We know from the third and the second lines of equation \eqref{5.4} that $u|_{\partial\omega_{s}}$ is a linear combination of $\{\psi^{\alpha}\}$. Multiplying the first line of equation \eqref{5.4} by $u$ and integrating by parts leads to, using a version of the Second Korn inequality (Lemma \ref{lemmaD}),
$$0=\int_{\Omega\setminus\overline{\omega}}\left(\mathbb{C}^{(0)}e(u),e(u)\right)dx\geq
\frac{1}{C}\int_{\Omega\setminus\overline{\omega}}|e(u)|^{2}dx\geq
\frac{1}{C}\int_{\Omega\setminus\overline{\omega}}|\nabla{u}|^{2}dx.
$$
It follows that $u=0$.
\end{proof}

The existence of a solution
 can be obtained by using the variational method.

Define the energy functional
\begin{equation}\label{5.5}
I_{\infty}[v]:=\frac{1}{2}\int_{\Omega\setminus\overline{\omega}}\left(\mathbb{C
}^{(0)}e(u),e(u)\right)dx,
\end{equation}
where $v$ belongs to the set
$$\mathcal{A}:=\left\{u\in{H}_{\varphi}^{1}(\Omega;\mathbb{R}^{d})
~\big|~e(u)=0~~\mbox{in}~\omega\right\}.$$
\begin{theorem}\label{thm_minimizer_infity}
There exists a minimizer $u\in\mathcal{A}$ satisfying
$$I_{\infty}[u]=\min_{v\in\mathcal{A}}I_{\infty}[v].$$
Moreover, $u\in{H}^{1}(\Omega
;\mathbb{R}^{d}
)\cap{C}^{1}(\overline{\Omega\setminus\omega}
;\mathbb{R}^{d}
)\cap{
C}^{1}(\overline{\omega}
;\mathbb{R}^{d}
)$ is a solution of equation \eqref{5.4}.
\end{theorem}

\begin{proof}
By the lower semi-continuity of $I_{\infty}$ and the weakly closed property of $\mathcal{A}$, it is not difficult to see that a minimizer $u\in\mathcal{A}$ exists and satisfies $\nabla\cdot(\mathbb{C}^{(0)}e(u))=0$ in $\Omega\setminus\overline{\omega}$. The only thing needs to shown is the fourth line of \eqref{5.4}, i.e.
$$\int_{\partial\omega_{s}}\dfrac{\partial{u}}{\partial\nu_{0}}\Big|_{+}\cdot\psi^{\alpha}=0,\quad1\leq{s}\leq{m}.$$
Indeed, since $u$ is a minimizer, for any $1\leq{s}\leq{m}$,
$1\le \alpha\le d(d+1)/2$, and any $\phi\in{C}^{\infty}_{c}(\Omega
;\mathbb{R}^{d}
)$ satisfying $\phi\equiv\psi^{\alpha}$ on $\overline{\omega}_{s}$ and $\phi=0$ on $\overline{\omega}_{t}$ $(t\neq{s})$, let
$$i(t):=I_{\infty}[u+t\phi],\quad\,t\in\mathbb{R},$$
we have
$$0=i '(0):=\frac{di}{dt}\big|_{t=0}=\int_{\Omega\setminus\overline{\omega}}\left(\mathbb{C}^{(0)}e(u),e(\phi)\right)dx.$$
Therefore
\begin{eqnarray*}
0&=&-\int_{\Omega\setminus\overline{\omega}}
\nabla\cdot\left(\mathbb{C}^{(0)}e(u)\right)\cdot\phi dx
=\int_{\Omega\setminus\overline{\omega}}\left(\mathbb{C}^{(0)}e(u),e(\phi)
\right) dx
+\int_{\partial\omega_{s}}\frac{\partial{u}}{\partial\nu_{0}}\big|_{+}\cdot\phi
\\
&=&\int_{\partial\omega_{s}}\frac{\partial{u}}{\partial\nu_{0}}\big|_{+}\cdot
\psi^{\alpha}.
\end{eqnarray*}
\end{proof}

Finally, we give the relationship between $u_{n}$ and $u$.

\begin{theorem}\label{thm_5.7}
Let $u_{n}$ and $u$ in $H^{1}(\Omega; \mathbb{R}^{d})$
be the solutions of equations \eqref{5.2} and \eqref{5.4}, respectively. Then
\begin{equation}
u_{n}
\to
{u}\quad\mbox{in}~H^{1}(\Omega; \mathbb{R}^{d}),~\mbox{as}~n\rightarrow\infty,
\label{abc1}
\end{equation}
and
\begin{equation}
\lim_{n\rightarrow\infty}I_{n}[u_{n}]=I_{\infty}[u],
\label{abc2}
\end{equation}
where $I_{n}$ and $I_{\infty}$ are defined by \eqref{5.3} and \eqref{5.5}.
\end{theorem}

\begin{proof}
{\bf Step 1.}
 Prove that $\{u_{n}\}$ weakly converges in $H^{1}(\Omega; \mathbb{R}^{d})$
to a solution  $u$
of   \eqref{5.4}.

Due to the uniqueness of the solution to  \eqref{5.4},
 we only need to show that after passing to a subsequence,
$\{u_n\}$ weakly converges in
$H^{1}(\Omega; \mathbb{R}^{d})$ to a solution $u$
 of \eqref{5.4}.

Let $\eta\in{H}^{1}_{\varphi}(\Omega; \mathbb{R}^{d})$ be
 fixed and satisfy
 $\eta\equiv0$ on $\overline{\omega}$.
Since
 $u_{n}$ is the minimizer of $I_{n}$ in $H^{1}_{\varphi}(\Omega
;\mathbb{R}^{d}
)$, we have, for some constant $C$ independent of $n$,
$$\frac{1}{C}\|e(u_{n})\|^{2}_{L^{2}(\Omega)}\leq\,I_{n}[u_{n}]\leq\,I_{n}[\eta]
=\frac{1}{2}\int_{\Omega\setminus\overline{\omega}}\left(\mathbb{C}^{(0)}e(\eta),e(\eta)\right)dx
\leq\,C\|\eta\|^{2}_{H^{1}(\Omega)}.$$
Using the Second Korn inequality and the fact that
$u_n=\varphi$ on $\partial \Omega$, we obtain
$$\|u_{n}\|_{H^{1}(\Omega)}\leq\,C,$$
and therefore, along a subsequence,
$$u_{n}\rightharpoonup{u}\quad\mbox{in}~H_{\varphi}^{1}(\Omega;
 \mathbb{R}^{d}),~\mbox{as}~n\rightarrow\infty.$$

Next we show that
 $u$ is a solution of equation \eqref{5.4}. In fact, we only need to prove the following three conditions:
\begin{align}
\nabla\cdot\left(\mathbb{C}^{(0)}e(u)\right)=0,\quad\quad&\mbox{in}~\Omega\setminus\overline{\omega},\label{5.6}\\
e(u)=0,\quad\quad&\mbox{in}~\omega,\label{5.7}\\
\int_{\partial\omega_{s}}\frac{\partial{u}}{\partial\nu_{0}}\big|_{+}\cdot\psi^{\alpha}=0,\quad\quad&1\leq{s}\leq{m},~
1\le \alpha\le d(d+1)/2.\label{5.8}
\end{align}

(i) Since $u_{n}\in{H}^{1}(\Omega
;\mathbb{R}^{d}
)$ is a solution of equation \eqref{5.1} and $u_{n}\rightharpoonup{u}$ in $H^{1}_{\varphi}(\Omega
;\mathbb{R}^{d}
)$, we have, for any $\phi\in{C}^{\infty}_{c}(\Omega\setminus\overline{\omega}
;\mathbb{R}^{d}
)$, that
$$0=\int_{\Omega\setminus\overline{\omega}}\left(\mathbb{C}^{(0)}e(u_{n}),e(\phi)\right)dx
\rightarrow\int_{\Omega\setminus\overline{\omega}}\left(\mathbb{C}^{(0)}e(u),e(\phi)\right)dx.$$
Therefore
$$\int_{\Omega\setminus\overline{\omega}}\left(\mathbb{C}^{(0)}e(u),e(\phi)\right)dx=0,
\quad\forall~\phi\in{C}^{\infty}_{c}(\Omega\setminus\overline{\omega}),$$
that is \eqref{5.6}.

(ii) Let $\eta\in{H}^{1}_{\varphi}(\Omega
;\mathbb{R}^{d}
)$ be fixed and satisfy $\eta\equiv0$ on $\overline{\omega}$, then since $u_{n}$ is a minimizer of $I_{n}$ in ${H}^{1}_{\varphi}(\Omega
;\mathbb{R}^{d}
)$, we have
$$I_{n}[u_{n}]\leq\,I_{n}[\eta]
\leq\frac{1}{2}\int_{\Omega\setminus\overline{\omega}}\left(\mathbb{C}^{(0)}e(\eta),e(\eta)\right)dx\leq\,C.$$
On the other hand,
$$I_{n}[u_{n}]\geq\sum_{s=1}^{m}\min\{2\mu_{n}^{(s)},d\lambda_{n}^{(s)}+2\mu_{n}^{(s)}\}\int_{\omega_{s}}|e(u_{n})|^{2}dx.$$
Since $\mu_{n}^{(s)}\rightarrow\infty$ and $d\lambda_{n}^{(s)}+2\mu_{n}^{(s)}\rightarrow\infty$ as $n\rightarrow\infty$, we have
$$\|e(u_{n})\|_{L^{2}(\omega)}\rightarrow0,\quad\mbox{as}~n\rightarrow\infty.$$
By (1), $u_{n}\rightharpoonup{u}$ in $H^{1}(\Omega;\mathbb{R}^{d})$. Therefore
$$\|e(u)\|_{L^{2}(\omega)}=0,$$
i.e. $e(u)=0$ in $\omega$, which is \eqref{5.7}.

(iii) By (i) and (ii), $u$ satisfies \eqref{5.6} and is a linear combination of $\{\psi^{\alpha}\}$ on each $\partial\omega_{s}$, and is equal to $\varphi$ on $\partial\Omega$. Thus $u$ is smooth on
 $\partial{\omega}$. By the elliptic regularity theorems,
 $u\in{C}^{1}(\overline{\Omega\setminus\omega}
;\mathbb{R}^{d}
)
\cap {C}^{2}(\Omega\setminus \overline \omega
;\mathbb{R}^{d}
)$.
 For each $s=1,2,\cdots,m$, $1\le \alpha
\le d(d+1)/2$,
we construct a function $\rho\in{C}^{2}(\overline{\Omega\setminus\omega};\mathbb{R}^{d})$ such that $\rho=\psi^{\alpha}$ on $\partial\omega_{s}$, $\rho=0$ on $\partial\omega_{t}$ for $t\neq{s}$, and $\rho=0$ on $\partial\Omega$. By Green's identity, we have the following:
\begin{align*}
0&=-\int_{\Omega\setminus\overline{\omega}}\nabla\cdot\left(\mathbb{C}^{(0)}e(u_{n})\right)\cdot\rho dx\\
&=\int_{\Omega\setminus\overline{\omega}}\left(\mathbb{C}^{(0)}e(u_{n}),e(\rho)\right)dx
+\int_{\partial\omega_{s}}\frac{\partial{u}_{n}}{\partial\nu_{0}}\big|_{+}\cdot\psi^{\alpha}\\
&=\int_{\Omega\setminus\overline{\omega}}\left(\mathbb{C}^{(0)}e(u_{n}),e(\rho)\right)dx
+\int_{\partial\omega_{s}}\frac{\partial{u}_{n}}
{\partial\nu_{0}}\big|_{-}\cdot\psi^{\alpha}\\
&=\int_{\Omega\setminus\overline{\omega}}
\left(\mathbb{C}^{(0)}e(u_{n}),e(\rho)\right)dx.
\end{align*}
Similarly,
\begin{align*}
0=-\int_{\Omega\setminus\overline{\omega}}\nabla\cdot\left(\mathbb{C}^{(0)}e(u)\right)\cdot\rho dx
=\int_{\Omega\setminus\overline{\omega}}\left(\mathbb{C}^{(0)}e(u),e(\rho)
\right)dx
+\int_{\partial\omega_{s}}\frac{\partial{u}}{\partial\nu_{0}}\big|_{+}
\cdot\psi^{\alpha}.
\end{align*}
Since $u_{n}\rightharpoonup{u}$ in $H^{1}(\Omega)$, it follows that
$$0=\int_{\Omega\setminus\overline{\omega}}\left(\mathbb{C}^{(0)}e(u_{n}),e(\rho)\right) dx
\rightarrow\int_{\Omega\setminus\overline{\omega}}\left(\mathbb{C}^{(0)}e(u),e(\rho)\right)dx.
$$
Thus
$$\int_{\partial\omega_{s}}\frac{\partial{u}}{\partial\nu_{0}}\big|_{+}\cdot\psi^{\alpha}=0,\quad\,1\leq{s}\leq{m},~ 1\le \alpha\le d(d+1)/2.$$
Step 1 is completed.

{\bf Step 2.} Prove (\ref{abc1}) and (\ref{abc2}).

Since $u_{n}$ is a minimizer of $I_{n}$ and $e(u)=0$ in $\omega$, we have
$$I_{n}[u_{n}]\leq\,I_{n}[u]=\,I_{\infty}[u].$$
Thus
$$\limsup_{n\rightarrow\infty}I_{n}[u_{n}]\leq\,I_{\infty}[u].$$
On the other hand,
since $e(u)=0$ and $u_{n}\rightharpoonup{u}$ in $H^{1}(\Omega; \mathbb{R}^{d})$,
\begin{eqnarray*}
I_{\infty}[u]&=&\frac{1}{2}\int_{\Omega\setminus\overline{\omega}}\left(\mathbb{C}^{(0)}e(u),e(u)\right)dx
\\
&\leq&
\liminf_{n\rightarrow\infty}\frac{1}{2}
\int_{\Omega\setminus\overline{\omega}}\left(\mathbb{C}^{(0)}e(u_{n}),e(u_{n})\right)dx
\\
&\leq&
\liminf_{n\rightarrow\infty}\frac{1}{2}
\int_{\Omega\setminus\overline{\omega}}\left(\mathbb{C}^{(0)}e(u_{n}),e(u_{n})\right)dx
+
\limsup_{n\rightarrow\infty}\frac{1}{2}
\sum_{s}\int_{\omega_s}\left(\mathbb{C}_n^{(s)}e(u_{n}),e(u_{n})\right)dx
\\
&\leq&
\limsup_{n\rightarrow\infty}I_{n}[u_{n}].
\end{eqnarray*}
With the help of the first Korn$'$s inequality, we easily deduce
(\ref{abc2}) and (\ref{abc1}) from the above.
The proof of Theorem \ref{thm_5.7} is completed.
\end{proof}

\noindent{\bf{\large Acknowledgments.}} The first author was
partially supported by NNSF (11071020) (11371060), SRFDPHE (20100003110003),  and Beijing Municipal Commission of Education for the Supervisor of Excellent Doctoral Dissertation (20131002701). The second author was partially supported by SRFDPHE (20100003120005) and NNSF (11201029).
The research of the third author is partially supported by NSF grant
DMS-1065971, DMS-1203961.
All authors were partially supported by the Fundamental Research Funds
for the Central Universities. We thank an anonymous referee for helpful suggestions which improve the exposition.

\vspace{5mm}


\bibliographystyle{amsplain}
\bibliography{References}

\end{document}